\documentclass[11pt,leqno,oneside,letterpaper]{amsart}
\usepackage[letterpaper,left=1.75truein, top=1truein]{geometry}
\usepackage{amsmath,amsfonts,amssymb,amscd,amsthm,amsbsy,epsf,graphics,mathtools,mathrsfs}
\usepackage{xy}
\xyoption{all}

\usepackage{color}

\usepackage[colorlinks,linkcolor=blue,citecolor=blue]{hyperref}
\usepackage{epsfig}
\usepackage{graphicx}
\usepackage{ifpdf}
\usepackage{epstopdf}

\newtheorem{thm}{Theorem}[section]

\newtheorem{lemma}[thm]{Lemma}
\newtheorem{prop}[thm]{Proposition}

\newtheorem{exa}[thm]{Example}
\newtheorem{que}[thm]{Question}

\newtheorem{defin}[thm]{Definition}

\theoremstyle{definition}

\def\zed{{\mathbb Z}}

\newtheoremstyle{cases}
  {12pt plus 6 pt}
  {2pt}
  {\bfseries}   
  {}
  {\bfseries}
  {.}
  {.5em}
  {}
\theoremstyle{cases}

\newtheorem{case}{Case}

\numberwithin{subcase}{section}
\numberwithin{equation}{section}

\def\sfrac#1#2{\kern.1em\raise.5ex\hbox{$#1$}
    \kern-.1em/\kern-.05em\lower.25ex\hbox{$#2$}}

\def\A{{\mathcal A}}

\def\T{{\mathcal T}}
\def\wT{\widetilde{\T}}

\def\zed{{\mathbb Z}}
\def\wA{\widetilde{\A}}

\def\fl{\mathfrak{\lambda}}

\def\zed{{\mathbb Z}}

\def\sfrac#1#2{\kern.1em\raise.5ex\hbox{$#1$}
        \kern-.1em/\kern-.05em\lower.25ex\hbox{$#2$}}

\def\T{{\mathcal T}}

\def\zed{{\mathbb Z}}

 \def\g{{\gamma}}

 \def\c{{\mathbb C}}
 \def\complex{{\mathbb C}}
\def\naturals{{\mathbb N}} 
\def\integers{{\mathbb Z}}
\def\z{{\mathbb Z}}
 \def\2{{\mathbb Z_2}}
 
 \def\t{{\tau}}

 \def\sl2{{SL(2,\mathbb C)}}

 \def\sm{{{\mbox{\small M}}}}
  \def\sl{{{\mbox{\small L}}}}
\def\sc{{{\mbox{\tiny C}}}}

\def\sk{{{\mbox{\tiny K}}}}
\def\su{{{\mbox{\tiny U}}}}
\def\se{{{\mbox{\tiny E}}}}

\begin{document}

\title{The AJ-conjecture and cabled knots over the figure eight knot }

\author{Dennis Ruppe }
\address{Department of Mathematics, University at Buffalo, Buffalo, NY, 14214-3093, USA.}
\email{dennisru@buffalo.edu}

\begin{abstract}
We show that most cabled knots over the figure eight knot in $S^3$ satisfy the $AJ$-conjecture,
in particular, any $(r,s)$-cabled knot over the figure eight knot satisfies
the $AJ$-conjecture if $r$ is not a number between
$-4s$ and $4s$.  
\end{abstract}

\maketitle

\section{Introduction}
For a knot $K$ in $S^3$, let $J_{\sk, n}(t)$ denote the \textit{$n$-colored Jones polynomial} of $K$
with the zero framing, normalized so that for the unknot $U$,
\[J_{\su,n}(t) = \frac{t^{2n}-t^{-2n}}{t^2-t^{-2}}.\]
The colored Jones polynomial is a powerful quantum invariant that has many surprising connections to classical invariants. For example, the Melvin-Morton-Rozansky conjecture, proved in \cite{BNG}, states that the Alexander-Conway polynomial of a knot can be recovered from a certain limit of the colored Jones polynomials. The volume conjecture, raised in \cite{K} and put in terms of the colored Jones polynomial in \cite{MM}, is an important and still open conjecture relating the polynomial to the hyperbolic volume of the knot complement. In this paper, we investigate the $AJ$-conjecture, raised in \cite{G}, which relates recurrence relations of the sequence of colored Jones polynomials to the $A$-polynomial.

For every knot $K$, it was proven in \cite{GaLe} that the sequence of colored Jones polynomials $J_{\sk,n}(t)$ satisfies a nontrivial recurrence relation. By defining $J_{\sk,-n}(t) := -J_{\sk,n}(t)$ for $n \in \integers$, we can treat $J_{\sk,n}(t)$ as
 a discrete function  $$J_{\sk,-}(t): \z \to \z[t^{\pm 1}].$$
 The quantum torus
\[ \T = \c[t^{\pm 1}] \left<\sl^{\pm 1}, \sm^{\pm 1} \right> / (\sl \sm-t^2\sm \sl)\]
acts on   the set of discrete functions $f: \z \to \c[t^{\pm 1}]$ by
\[(\sm f)(n) := t^{2n}f(n), \quad  (\sl f)(n) := f(n+1).\]
A linear homogeneous recurrence relation $\sum_{i=0}^d P_i(t,\sm)J_{\sk,n+i}(t) = 0$ then corresponds to a polynomial $P(t,\sm,\sl) \in \T$ that is an annihilator of $J_{\sk,n}(t)$ and vice versa.
The set of all such annihilators
  $${\mathcal A}_\sk := \{P(t,\sm,\sl) \in \T \mid P(t,\sm,\sl) J_{\sk,n}(t) = 0\},$$
which is a left ideal of $\T$, is called the \textit{recurrence ideal} of $K$. For every knot $K$, we know that ${\mathcal A}_\sk$ is nontrivial.

The ring $\T$ can be extended to a principal left ideal domain $\wT$ by including inverses of polynomials in $t$ and $\sm$. In $\wT$, we have a product defined by
\[f(t,\sm) \sl^a \cdot g(t,\sm) \sl^b = f(t,\sm) g(t,t^{2a}\sm)\sl^{a+b}\]
for any rational functions $f(t,\sm)$, $g(t,\sm)$ in $\complex(t,\sm)$. The left ideal $\wA_\sk=\wT {\mathcal A_\sk}$ is then generated by some nonzero polynomial in $\wT$, and in particular, this generator can be chosen to be in ${\mathcal A}_\sk$ and be of the form
\[\alpha_\sk(t,\sm,\sl) = \sum_{i=0}^d P_i \sl^i,\]
with $d$ minimal and with  $P_1,...,P_d \in \z[t,\sm]$ being relatively prime in $\z[t,\sm]$. This polynomial $\alpha_\sk$ is uniquely determined up to a sign and is called the (\textit{normalized}) \textit{recurrence polynomial} of $K$.

The $A$-polynomial was introduced in \cite{CCGLS}. For a knot $K$ in $S^3$,
its $A$-polynomial $A_\sk(\sm,\sl)\in \z[\sm,\sl]$ is a two variable polynomial
with no repeated factors and with relatively prime integer coefficients,
which is uniquely associated to $K$ up to a sign. Note that $A_\sk(\sm,\sl)$ always contains the factor $\sl-1$.

The $AJ$-conjecture states that for every knot $K$, its recurrence polynomial $\alpha_\sk(t,\sm,\sl)$  evaluated at $t = -1$ is equal to the $A$-polynomial of $K$, up to a factor of a polynomial in $\sm$.
So far, it has been shown that torus knots, some classes of $2$-bridge knots and pretzel knots, and most cabled knots over torus knots satisfy the conjecture; see  \cite{G}, \cite {Ta}, \cite{Hikami},
\cite{Le}, \cite{LT}, \cite{Tran}, \cite{RZ}.

In our previous work \cite{RZ}, we used explicit formulas to verify the $AJ$-conjecture for most cabled knots over torus knots. For the figure eight knot, its colored Jones polynomials and $A$-polynomial are more complicated, and the $A$-polynomials of the cabled knots over the figure eight knot are given in terms of the resultant of two polynomials, requiring a more theoretical connection. 

\begin{thm}\label{main result}
The $AJ$-conjecture holds for each $(r,s)$-cabled knot $C$ over the figure eight knot $E$ when $r > 4s$ or $r < -4s$.
\end{thm}

A cabling formula for $A$-polynomials of cabled knots
 in $S^3$ is  given in \cite{NZ}. In particular, when
 $C$ is  the $(r, s)$-cabled knot over the figure eight knot $E$ in $S^3$,  its
 $A$-polynomial $A_\sc(\sm, \sl)$  is given in terms of the $A$-polynomial $A_\se(\sm,\sl)$ of the figure eight knot.
For  a pair of relatively prime integers $(r,s)$  with $s\geq 2$, define $F_{(r,s)}(\sm,\sl) \in \z[\sm,\sl]$ by:
$$F_{(r, s)}(\sm, \sl):= \left\{\begin{array}
 {ll}\sm^{2r}\sl+1,\;\;&  \mbox{if $s=2$, $r>0$,}\\
\sl+\sm^{-2r},\;\;  &\mbox{if $s=2$,  $r<0$,}\\
\sm^{2rs}\sl^2-1,\;\;&  \mbox{if $s>2$, $r>0$,}\\
\sl^2-\sm^{-2rs},\;\; & \mbox{if $s>2$, $r<0$}\end{array}\right.$$
Then
\begin{equation}\label{A-poly of C}A_\sc(\sm,\sl)=
(\sl-1)F_{(r, s)}(\sm, \sl) Red(Res_\fl(\frac{A_{\se}(\sm^{s},\fl)}{\fl-1},\fl^s-\sl)),
\end{equation}
where $Red$ denotes the function reducing polynomials by eliminating repeated factors and $Res_\fl$ denotes the polynomial resultant eliminating the variable $\fl$; see Section \ref{subsec: Resultant} for a definition. The $A$-polynomial $A_{\se}(\sm,\sl)$ of the figure eight knot $E$, is
$$A_\se (\sl,\sm) = (\sl-1)(-\sl+\sl \sm^2+ \sm^4+2\sl \sm^4+\sl^2\sm^4+\sl \sm^{6}-\sl \sm^{8}).$$
Meanwhile, its colored Jones polynomial, as given in \cite{G} but in our normalized form, is
$$J_{\se,n}(t) = \frac{t^{2n}-t^{-2n}}{t^2-t^{-2}}\sum_{k=0}^{n-1} \prod_{i=1}^{k}((t^{2n}-t^{-2n})^2-(t^{2i}-t^{-2i})^2).$$

The figure eight knot $E$ has an inhomogeneous recurrence polynomial $(\tilde{\alpha}_\se(t,\sm,\sl),b(t,\sm))$, found in \cite{G} and changed here to suit our normalization, given by
\begin{equation} \label{eq: E Inhomogeneous}
\begin{array}{rl}
\tilde{\alpha}_\se(t,\sm,\sl)J_{\se,n}(t) &= (P_2(t,\sm)\sl^2 + P_1(t,\sm)\sl + P_0(t,\sm))J_{E,n}(t) = b(t,\sm), \text{ where}\\
P_2(t,\sm) &= t^{10}\sm^4(-1+t^4\sm^4)\\
P_1(t,\sm) &= -(-1+t^4\sm^2)(1+t^4\sm^2)(1-t^4\sm^2-t^4\sm^4 -t^{12}\sm^4-t^{12}\sm^6+t^{16}\sm^8)\\
P_0(t,\sm) &= t^{6}\sm^4 (-1+t^{12}\sm^4)\\
b(t,\sm) &= \frac{\sm (1+t^4\sm^2)(-1+t^4\sm^4)(-t^2+t^{14}\sm^4)}{t^2-t^{-2}}.
\end{array}
\end{equation}
Notice that any inhomogeneous recurrence relation gives rise to a homogeneous one, since if $P(t,\sm,\sl)J_{\sk,n}(t) = b(t,\sm)$ for $b(t,\sm) \neq 0$, then
$$(\sl-1)b(t,\sm)^{-1}P(t,\sm,\sl)J_{\sk,n}(t) = 0,$$
and by multiplying by a suitable polynomial in $t$ and $\sm$, we can recover an annihilator in $\mathcal{A}_\sk$. 

A cabling formula for the $n$-colored Jones polynomial of
the $(r,s)$-cabled knot $C$ over a knot $K$ is given in \cite{Morton} (see also
\cite{Veen}) which in our normalized form is:
\begin{equation}\label{cabling formula of colored Jones}
 J_{\sc,n}(t) =
t^{-rs(n^2-1)}\sum_{k=-\frac{n-1}{2}}^{\frac{n-1}{2}}
t^{4rk(ks+1)}J_{\sk,2ks+1}(t).
\end{equation}
We divide the proof Theorem \ref{main result} into two cases:
\begin{enumerate}
\item $s = 2$;
\item $s > 2$.
\end{enumerate}
In each case, we use a relation for the cabling formula (\ref{cabling formula of colored Jones}) and an inhomogeneous recurrence relation of $J_{\se,n}(t)$ to obtain an annihilator of $J_{\sc,n}(t)$, and then proceed to prove that it is the recurrence polynomial $\alpha_\sc(t,\sm,\sl)$ of $C$ when $r$ is not between $-4s$ and $4s$ by using formulas for the degrees of the colored Jones polynomials given in Section \ref{sec: Formulas}. For the case $s=2$, we directly compute $A_\sc(\sm,\sl)$ and $\alpha_\sc(-1,\sm,\sl)$ to verify the $AJ$-conjecture, while for the case $s > 2$, we verify the $AJ$-conjecture by exploring the relationship between our annihilator and the resultant.

In practice, we find a minimal degree annihilator in $\wA_\sc$ of the form $P = \sum_{i=0}^d P_i L^i$, which is equal to the normalized recurrence polynomial up to a rational function $C(t,\sm)$, which is enough for the purpose of verifying the $AJ$-conjecture. Also note that changing the sign of $r$ only changes the $A$-polynomial of $C$ up to a power of $\sm$, which is of no consequence when checking up to a factor of a rational function in $\sm$.

\subsection{Acknowledgments} The author would like to thank Xingru Zhang for numerous discussions and the referee for many helpful comments. Anh T. Tran \cite{TranCable} has independently obtained similar results for the $(r,2)$-cables of the figure eight knot.

\section{Degree Formulas and Preliminaries} \label{sec: Formulas}
For the rest of the paper, let $E$ denote the figure eight knot and $C$ the $(r,s)$-cabled knot over $E$ unless specified otherwise. Also, we often write $J_{\sk,n}$ rather than $J_{\sk,n}(t)$ for brevity.

For a polynomial $f(t) \in \z[t^{\pm 1}]$, let $\ell[f]$ and $\hbar[f]$ denote the lowest degree and highest degree of $f$ in $t$ respectively. For $f(t), g(t) \in \z[t^{\pm 1}]$, these functions satisfy $\ell[fg] = \ell[f]+\ell[g]$ and $\hbar[fg] = \hbar[f]+\hbar[g]$.

The degrees of the colored Jones polynomials of alternating knots are known \cite[Proposition 2.1]{Le}. For a non-trivial knot $K$ with reduced alternating diagram $D$ having $k$ crossings and writhe $w$, if $n > 0$, then
\begin{align*}
\hbar[J_{\sk,n}] &= k(n-1)^2-w(n^2-1)+2(n-1)s_+(D),\\
\ell[J_{\sk,n}] &= -k(n-1)^2-w(n^2-1)-2(n-1)s_-(D),
\end{align*}
where $\hbar$ and $\ell$ denote the highest and lowest degrees in $t$, respectively, and $s_+(D)$, $s_-(D)$ are the number of circles obtained by positively or negatively smoothing the crossings of $D$. Notice that $s_-(D)+s_+(D) = k+2$, so $2 \leq s_-(D), s_+(D) \leq k$. In particular, the figure eight knot has a reduced alternating diagram with 4 crossings, $s_+(D) = s_-(D) = 3$, and zero writhe, so we have the following lemma.

\begin{lemma}\label{lemma: E Degree}
Let $E$ be the figure eight knot. Then for all $n \neq 0$,
\begin{align*}
\hbar[J_{E,n}] &= 4(|n|-1)^2+6(|n|-1) = 4n^2-2|n|-2,\\
\ell[J_{E,n}] &= -4(|n|-1)^2-6(|n|-1) = -4n^2+2|n|+2.
\end{align*}
\end{lemma}
The following lemma gives the degrees of $J_{\sc,n}(t)$ for any cabled knot $C$ over an alternating knot $K$.
\begin{lemma}\label{lemma: C Degree}
Let $C$ be the $(r,s)$-cable knot over a knot $K$ with a reduced alternating diagram $D$ with $N$ crossings and $s_-(D) = m$, $s_+(D) = N+2-m$. Then for $n > N$, we have
\begin{align*}
\hbar[J_{C,n}] &= \begin{cases} \!\! \begin{array}{l}(N-w)s^2 n^2 + (2r -2(-2+m+r+w-N)s-2(N-w)s^2) n\\
																 \quad+(-2r+2(-2+m+r+w-N)s+(N-w)s^2),\end{array}  & r> -(N-w)s\\
                                 -rs(n^2-1)+\frac{1}{2}(1-(-1)^{n-1})(s-2)(4+r+(N-w)s-2m), & r < -(N-w)s\end{cases} \\
\ell[J_{C,n}] &= \begin{cases}  \!\! \begin{array}{l}-(N+w)s^2n^2+(2r-2(r+m+w)s+2(N+w)s^2)n \\
																\qquad +(-2r+2(r+m+w)s-(N+w)s^2),\end{array} & r < (N+w)s \\
                                -rs(n^2-1)+\frac{1}{2}(1-(-1)^{n-1})(s-2)(r-(N+w)s+2(N-m)), & r > (N+w)s \end{cases}
\end{align*}
In particular, for the cabled knot $C$ over the figure eight knot, we set $N = 4$, $m = 3$, and $w = 0$:
\begin{align*}
\hbar[J_{C,n}] &= \begin{cases}  4s^2n^2+(2r+6s-2rs-8s^2)n+(-2r-6s+2rs+4s^2),  & r> -4s\\
                                 -rs(n^2-1)+\frac{1}{2}(1-(-1)^{n-1})(s-2)(-2+r+4s), & r < -4s\end{cases} \\
\ell[J_{C,n}] &= \begin{cases}  -4s^2n^2+(2r-6s-2rs+8s^2)n+(-2r+6s+2rs-4s^2), & r < 4s \\
                                -rs(n^2-1)+\frac{1}{2}(1-(-1)^{n-1})(s-2)(2+r-4s), & r > 4s \end{cases}
\end{align*}
\end{lemma}
The restriction $n > N$ can usually be relaxed. For the figure eight knot case, the formulas hold for all $n > 0$.
\begin{proof}
We know from the cabling formula
$$ \ell[J_{C,n}] = -rs(n^2-1) + \min\{\ell[J_{\sk,2sk+1}] + 4rk(sk+1) \mid -\tfrac{n-1}{2} \leq k \leq \tfrac{n-1}{2}\}$$
where $k$ is integer or half-integer valued. Let $g(k) = \ell[J_{E,2sk+1}] + 4rk(sk+1)$. From the discussion above, we know
\begin{align*}
g(k) &= -N(|2sk+1|-1)^2-w((2sk+1)^2-1)-2(|2sk+1|-1)m +4rk(sk+1)\\
&=\begin{cases} 4s k^2  (r - (N+w)s) + 4 k (r - (m+w)s), & k \geq 0\\
                4s k^2(r-(N+w)s)+4k(r-(2N-m+w)s)-4(N-m), & k \leq -\frac{1}{2} \end{cases}
\end{align*}
Since $r$ and $s$ are relatively prime, $r \neq (N+w)s$, so each piece of $g(k)$ is quadratic with critical points $k = \frac{r-(m+w)s}{-2s(r-(N+w)s)}$ and $k = \frac{r-(2N-m+w)s}{-2s(r-(N+w)s)}$ respectively. Notice that when $r > (N+w)s$, these critical points are local minima. Moreover, $2 \leq m \leq N$, so when $r > (N+w)s$, we have $r > (m+w)s$, so the first point is negative. Meanwhile, the second point is
$$\frac{r-(2N-m+w)s}{-2s(r-(N+w)s)} = \frac{r-(N+w)s}{-2s(r-(N+w)s)} + \frac{(-N+m)s}{-2s(r-(N+w)s)} = \frac{-1}{2s} + \frac{N-m}{2(r-(N+w)s)}  > \frac{-1}{2},$$
so each component of $g(k)$ is minimized at their endpoints $k = 0$ and $k = -1/2$. Then $g(0) = 0$ and $g(\frac{-1}{2}) = (s-2)(r-(N+w)s+2(N-m))$. The point $k = 0$ is attained when $n$ is odd, and $k = \frac{-1}{2}$ occurs when $n$ is even, so this gives us the formula
$$\ell[J_{C,n}] = -rs(n^2-1)+\frac{1}{2}(1-(-1)^{n-1})(s-2)(r-(N+w)s+2(N-m))$$
when $r > (N+w)s$, for all $n > 0$. 

In the case where $r < (N+w)s$, $g(k)$ must be minimized at either $k = \frac{-(n-1)}{2}$ or $k = \frac{n-1}{2}$. For $n > 1$, we see that
$$g(\tfrac{n-1}{2}) - g(\tfrac{-(n-1)}{2}) = 4(r-(N+w)s)(n-1)+4(N-m),$$
which is negative for $n > N$, so $g(\frac{n-1}{2})$ is smaller. Therefore
$$\ell[J_{C,n}] =-(N+w)s^2n^2+(2r-2(r+m+w)s+2(N+w)s^2)n+(-2r+2(r+m+w)s-(N+w)s^2).$$
The proof for $\hbar[J_{C,n}]$ is similar.
\end{proof}

The following lemma is a generalization of Proposition 2.2 in \cite{Le}.
\begin{lemma} \label{lemma: L>2}
Let $f: \z \to \complex[t^{\pm 1}]$ be a discrete function. Suppose $f$ satisfies the following conditions:
\begin{enumerate}
\item There exists an integer $k$ and a nonzero $c \in \complex$ such that for all $n$, $f(-n) = c f(n+k)$,
\item $f$ satisfies a nontrivial homogeneous recurrence relation, and
\item There exists an integer $N$ such that for all $n > N$, $\hbar[f(n)]-\ell[f(n)]$ is not a linear function in $n$.
\end{enumerate}
Then any nontrivial recurrence relation of $f$ has order at least 2.
\end{lemma}
\begin{proof}
Suppose $f: \z \to \z[t^{\pm 1}]$ satisfies the given hypotheses, and assume toward a contradiction that there is homogeneous recurrence relation of $f$ of order 1; that is, there exist nonzero $P_1(t,t^{2n})$, $P_0(t,t^{2n})$ such that $P_1(t,t^{2n}) f(n+1) + P_0(t,t^{2n}) f(n) = 0$. We can assume without loss of generality that $P_1(t,t^{2n})$ and $P_0(t,t^{2n})$ share no common factors. Then the annihilator ideal $\wA_f$ is nonzero and moreover is generated by $P_1(t,\sm) \sl + P_0(t,\sm)$ since it is a principal left ideal.

Substituting $-n$ for $n$, we know
\begin{align*}
0 &= P_1(t,t^{-2n}) f(-n+1) + P_0(t,t^{-2n}) f(-n)\\
&= P_1(t,t^{-2n}) c f(n+k-1) + P_0(t,t^{-2n}) c f(n+k).
\end{align*}
Canceling $c$ and shifting by replacing $n$ with $n-k+1$, we have
$$0 = P_1(t,t^{-2n+2k-2}) f(n) + P_0(t,t^{-2n+2k-2}) f(n+1),$$
so $P_0(t,t^{2k-2} \sm^{-1}) \sl + P_1(t,t^{2k-2} \sm^{-1})$ is also in $\wA_f$. Then this is a multiple of our generator, so since they both have $\sl$-degree 1, there is some $B(t,\sm) \in \complex(t,\sm)$ such that 
$$B(t,\sm) (P_1(t,\sm) \sl + P_0(t,\sm)) = P_0(t,t^{2k-2} \sm^{-1}) \sl + P_1(t,t^{2k-2} \sm^{-1}).$$
This implies $B(t,\sm) P_1(t,\sm) = P_0(t,t^{2k-2} \sm^{-1})$ and $B(t,\sm) P_0(t,\sm) = P_1(t,t^{2k-2} \sm^{-1})$, so solving for $B(t,\sm)$,
$$B(t,\sm) =  \frac{P_0(t,t^{2k-2} \sm^{-1})}{P_1(t,\sm)} = \frac{P_1(t,t^{2k-2} \sm^{-1})}{P_0(t,\sm)}.$$
This tells us
$$P_0(t,\sm) P_0(t,t^{2k-2} \sm^{-1}) = P_1(t,\sm) P_1(t,t^{2k-2} \sm^{-1}).$$
Since $P_0(t,\sm)$ and $P_1(t,\sm)$ are relatively prime, we must have $P_0(t,\sm)$ divides $P_1(t,t^{2k-2}\sm^{-1})$. Likewise, $P_0(t,t^{2k-2} \sm^{-1})$ and $P_1(t,t^{2k-2} \sm^{-1})$ are relatively prime, so we conclude that $P_0(t,\sm) = P_1(t,t^{2k-2} \sm^{-1})$. Substituting this into our original relation, we see
$$P_1(t,t^{2n}) f(n+1) + P_1(t,t^{2k-2-2n}) f(n) = 0$$
and thus
$$f(n+1) = -\frac{P_1(t,t^{2k-2-2n})}{P_1(t,t^{2n})} f(n).$$
Consider the degrees in $t$ of both sides. We see that for $n$ sufficiently large, the difference in breadths of $P_1(t,t^{2n})$ and $P_1(t,t^{2k-2-2n})$ is a constant $K$. This implies that $\hbar[f(n)] - \ell[f(n)]$ is a linear function in $n$ for large enough $n$, contrary to our assumption. We conclude that any recurrence relation of $f$ must have order at least 2.
\end{proof}

\section{Case $s = 2$}
 In this section, we prove the $s = 2$ case of Theorem \ref{main result} in three steps: we first find a polynomial that annihilates $J_{\sc,n}$, then verify the $AJ$-conjecture by evaluating our polynomial at $t = -1$, and finally we prove that our annihilator is of minimal order and hence the recurrence polynomial.

\subsection{An annihilator of the colored Jones polynomial}

In \cite[Lemma 3.1]{RZ}, the formula
$$J_{\sc,n+2}-t^{-4rsn-4rs}J_{\sc,n} = t^{2(r-rs)n-2rs+2r}J_{\sk,s(n+1)+1} - t^{2(-r-rs)n-2rs-2r}J_{\sk,s(n+1)-1}$$
is derived for any $(r,s)$-cabled knot $C$ over any knot $K$. Rearranging and changing to operator notation, this is
\begin{equation} \label{eq: s>2 peel}
(t^{2rs}\sm^{rs} \sl^2-t^{-2rs}\sm^{-rs})J_{\sc,n} = t^{2r}\sm^{r}J_{\sk,s(n+1)+1} - t^{-2r}\sm^{-r}J_{\sk,s(n+1)-1}.
\end{equation}
When $s = 2$, it is computed in \cite[Equation 6.1]{RZ} that
\begin{equation}J_{\sc,n+1} = -t^{-4rn-2r}J_{\sc,n}+ t^{-2rn}J_{\sk,2n+1},\end{equation}
which can be rewritten as
\begin{equation} \label{eq: s=2 peel}
(\sm^r\sl+t^{-2r}\sm^{-r})J_{\sc,n} = J_{\sk,2n+1}.
\end{equation} \label{eq: s=2 J_C(n+1)}
Therefore, to find an annihilator of $J_{\sc,n}$, it is enough to find an annihilator of $J_{\sk,2n+1}$.

\begin{lemma} \label{lemma: fig8 annihilator} Let $\tilde{\alpha}_\se(t,\sm,\sl) = P_2(t,\sm)\sl^2 + P_1(t,\sm)\sl + P_0(t,\sm)$ be the inhomogeneous recurrence polynomial defined in equation (\ref{eq: E Inhomogeneous}). Then $J_{\se,2n+1}(t)$ has an inhomogeneous recurrence relation given by the polynomial $Q(t,\sm,\sl) = Q_2(t,\sm)\sl^2+Q_1(t,\sm)\sl + Q_0(t,\sm)$, where
\begin{align*}
Q_2(t,\sm) &= P_2(t,t^4 \sm^2) P_1(t,t^2 \sm^2) P_2(t,t^6 \sm^2),\\
Q_1(t,\sm) &= P_0(t,t^4 \sm^2) P_1(t,t^6 \sm^2)P_2(t,t^2 \sm^2) - P_1(t,t^6 \sm^2) P_1(t,t^2 \sm^2) P_1(t,t^4 \sm^2)\\
&\qquad +P_2(t,t^4 \sm^2) P_1(t,t^2 \sm^2) P_0(t,t^6 \sm^2),\\
Q_0(t,\sm) &=P_0(t,t^4 \sm^2) P_1(t,t^6 \sm^2) P_0(t,t^2 \sm^2).
\end{align*}
\end{lemma}

\begin{proof}
Let $\tilde{\alpha}_\se(t,\sm,\sl) = P_2(t,\sm) \sl^2 + P_1(t,\sm) \sl + P_0(t,\sm)$ and $b(t,\sm)$ be given as in equation (\ref{eq: E Inhomogeneous}).
Changing $\sm$ to $t^{2n}$ for clarity, we have
$$P_2(t,t^{2n})J_{\se,n+2} + P_1(t,t^{2n}) J_{\se,n+1} + P_0(t,t^{2n}) J_{\se,n} = b(t,t^{2n}),$$
and substituting $2n+1$, $2n+2$, and $2n+3$ for $n$, we have
\begin{equation}\label{eq:2n relations}
\begin{array}{ll}
P_2(t,t^{4n+2})J_{\se,2n+3} + P_1(t,t^{4n+2}) J_{\se,2n+2} + P_0(t,t^{4n+2}) J_{\se,2n+1} &= b(t,t^{4n+2}),\\
P_2(t,t^{4n+4})J_{\se,2n+4} + P_1(t,t^{4n+4}) J_{\se,2n+3} + P_0(t,t^{4n+4}) J_{\se,2n+2} &= b(t,t^{4n+4}),\\
P_2(t,t^{4n+6})J_{\se,2n+5} + P_1(t,t^{4n+6}) J_{\se,2n+4} + P_0(t,t^{4n+6}) J_{\se,2n+3} &= b(t,t^{4n+6}).
\end{array}\end{equation}
A second degree inhomogeneous recurrence relation of $J_{\se,2n+1}(t)$ has the form
$$Q_2(t,t^{2n}) J_{\se,2n+5} + Q_1(t,t^{2n}) J_{\se,2n+3} + Q_0(t,t^{2n}) J_{\se,2n+1} = B(t,t^{2n})$$
for some rational functions $Q_i(t,t^{2n}), B(t,t^{2n})$.

We claim that we can find a linear combination of the relations (\ref{eq:2n relations}) that is of this form. That is, we want to solve
$$\sum_{j=0}^2  c_j \sum_{i=0}^2 P_i(t,t^{4n+2j+2}) J_{\se,2n+1+i+j} = \sum_{i=0}^2 Q_i(t,t^{2n}) J_{\se,2n+1+2i}$$
or equivalently
$$0 = \sum_{j=0}^2  c_j \sum_{i=0}^2 P_i(t,t^{4n+2j+2}) J_{\se,2n+1+i+j} - \sum_{i=0}^2 Q_i(t,t^{2n}) J_{\se,2n+1+2i}$$
for the unknown coefficients $c_0,c_1,c_2 \in \complex(t,\sm)$ and $Q_0,Q_1,Q_2$. It is enough to find the $c_i$'s and $Q_i$'s that make the coefficients on each $J_{\se,2n+1+i+j}$ vanish. Then we have the following system of equations:
$$\begin{array}{ll}
0 & = (c_0 P_0(t,t^{4n+2})-Q_0)J_{\se,2n+1}\\
0 & = (c_0 P_1(t,t^{4n+2}) +c_1 P_0(t,t^{4n+4})) J_{\se,2n+2}\\
0 & = (c_0 P_2(t,t^{4n+2}) +c_1 P_1(t,t^{4n+4}) +c_2 P_0(t,t^{4n+6})-Q_1)J_{\se,2n+3}\\
0 & = (c_1 P_2(t,t^{4n+4}) +c_2 P_1(t,t^{4n+6}))J_{\se,2n+4}\\
0 & = ( c_2 P_2(t,t^{4n+6}) -Q_2)J_{\se,2n+5}
\end{array}$$
Setting the coefficients equal to zero, we form a $5 \times 7$ augmented matrix:
$$\left[\begin{array}{cccccc|c}
 P_0(t,t^{4n+2}) &0 &0 & -1 &0 &0 & 0\\
 P_1(t,t^{4n+2}) & P_0(t,t^{4n+4}) &0 &0 &0 &0 & 0\\
 P_2(t,t^{4n+2}) & P_1(t,t^{4n+4}) & P_0(t,t^{4n+6}) &0 &-1 &0 & 0\\
 0& P_2(t,t^{4n+4}) & P_1(t,t^{4n+6}) &  0& 0&0 & 0\\
 0& 0&  P_2(t,t^{4n+6}) &0  &0 & -1 & 0
\end{array}\right]$$
where the columns correspond to $c_0,c_1,c_2,Q_0,Q_1,Q_2$ respectively. We can simply row-reduce this:
$$\renewcommand*{\arraystretch}{1.5}\setlength{\arraycolsep}{1pt} \left[\begin{array}{cccccc|c} 
1 & 0 & 0 & 0 & 0 & -\frac{P_0(t,t^{4n+4})P_1(t,t^{4n+6})}{P_1(t,t^{4n+2}) P_2(t,t^{4n+4}) P_2(t,t^{4n+6})} & 0\\
0 & 1 & 0 & 0 & 0 & \frac{P_1(t,t^{4n+6})}{ P_2(t,t^{4n+4}) P_2(t,t^{4n+6})} & 0\\
0 & 0 & 1 & 0 & 0 & -\frac{1}{ P_2(t,t^{4n+6})} & 0\\
0 & 0 & 0 & 1 & 0 & -\frac{P_0(t,t^{4n+2})P_0(t,t^{4n+4})P_1(t,t^{4n+6})}{P_1(t,t^{4n+2}) P_2(t,t^{4n+4}) P_2(t,t^{4n+6})} & 0\\
0 & 0 & 0 & 0 & 1 & \frac{P_1(t,t^{4n+2})P_1(t,t^{4n+4})P_1(t,t^{4n+6})-P_0(t,t^{4n+4})P_1(t,t^{4n+6})P_2(t,t^{4n+2})-P_0(t,t^{4n+6})P_1(t,t^{4n+2})P_2(t,t^{4n+4})}{P_1(t,t^{4n+2}) P_2(t,t^{4n+4}) P_2(t,t^{4n+6})} & 0
\end{array}\right]$$
Thus the system has a one-dimensional solution since $P_1(t,t^{4n+2})$ is nonzero. Then we can choose $Q_2 = P_1(t,t^{4n+2}) P_2(t,t^{4n+4}) P_2(t,t^{4n+6})$ to clear all of the denominators, which gives
\begin{align*}
Q_0 &= -P_0(t,t^{4n+2})P_0(t,t^{4n+4})P_1(t,t^{4n+6}),\\
Q_1 &= P_1(t,t^{4n+2})P_1(t,t^{4n+4})P_1(t,t^{4n+6})-P_0(t,t^{4n+4})P_1(t,t^{4n+6})P_2(t,t^{4n+2})\\
&\qquad -P_0(t,t^{4n+6})P_1(t,t^{4n+2})P_2(t,t^{4n+4}),\\
c_0 &= P_0(t,t^{4n+4}) P_1(t,t^{4n+6}),\\
c_1 &= -P_1(t,t^{4n+2}) P_1(t,t^{4n+6}),\\
c_2 &= P_1(t,t^{4n+2}) P_2(t,t^{4n+4}).
\end{align*}
Therefore, we have 
\begin{align*}
\sum_{i=0}^2 Q_i(t,t^{2n}) J_{\se,2n+1+2i} &= \sum_{j=0}^2  c_j \sum_{i=0}^2 P_i(t,t^{4n+2j+2}) J_{\se,2n+1+i+j} \\
&= \sum_{j=0}^2 c_j b(t,t^{4n+2j+2}).
\end{align*}
We can see that $B(t,\sm) := \sum_{j=0}^2 c_j b(t,t^{2j+2}\sm^{2})$ is nonzero by evaluating the limit as $t$ approaches $-1$ of $(t^2-t^{-2})B(t,\sm)$:
\begin{align*}
\lim_{t\to -1}(t^2-t^{-2})B(t,\sm) &= \lim_{t\to-1} (t^2-t^{-2})(P_0(t,t^{4}\sm^2) P_1(t,t^{6}\sm^2) b(t,t^{2}\sm^2) \\
&\qquad- P_1(t,t^{2}\sm^2) P_1(t,t^{6}\sm^2) b(t,t^{4}\sm^2) + P_2(t,t^{2}\sm^2) P_1(t,t^{4}\sm^2) b(t,t^{6}\sm^2))\\
&= (\sm (-1+\sm^4)(-1+\sm^4)) (P_0(1,\sm^2)P_1(1,\sm^2)-P_1(1,\sm^2)^2\\
&\qquad +P_1(1,\sm^2)P_2(1,\sm^2)) \\
&= (-1 + \sm)^6 \sm^2 (1 + \sm)^6 (1 + \sm^2)^6 (1 - \sm + \sm^2) (1 + \sm + \sm^2) (1 + \sm^4)^5\\
&\qquad  (-1 + \sm^2 - \sm^4) (1 - \sm^4 - 2 \sm^8 - \sm^{12} + \sm^{16}),
\end{align*}
which is indeed nonzero. Therefore $Q(t,\sm,\sl)$ is an inhomogeneous recurrence polynomial of $J_{\se,2n+1}$.
\end{proof}

Notice that our inhomogeneous recurrence relation $Q(t,\sm,\sl) (\sm^r\sl+t^{-2r}\sm^{-r}) J_{\sc,n} = B(t,\sm)$ found in Lemma \ref{lemma: fig8 annihilator} can be made into the homogeneous one 
$$R(t,\sm,\sl)J_{\sc,n} = (\sl-1)B(t,\sm)^{-1}Q(t,\sm,\sl)(\sm^r\sl+t^{-2r}\sm^{-r})J_{\sc,n} = 0.$$ 
Assuming $R(t,\sm,\sl)$ is the recurrence polynomial of $C$, we can check the $AJ$-conjecture by evaluating at $t = -1$. We can directly compute the $A$-polynomial of $C$ from equation (\ref{A-poly of C}):
\begin{align*}
A_\sc(\sm,\sl) &= (\sl-1)(\sm^{2r}\sl+1)(-\sl + 2 \sl \sm^4 + 3 \sl \sm^8 - 2 \sl \sm^{12} + \sm^{16} - 6 \sl \sm^{16} + \sl^2 \sm^{16} - 2 \sl \sm^{20} \\
&\qquad + 3 \sl \sm^{24} + 2 \sl \sm^{28} - \sl \sm^{32}).\end{align*}
And so
\begin{equation} \label{eq: s=2 AJ check}
\begin{array}{ll}
R(-1,\sm,\sl)= & (-1 + \sm)^3 (1 + \sm)^3 (1 + \sm^2)^3 (1 + \sm^4)^3 (1 - \sm^4 - 2 \sm^8 - \sm^{12} + \sm^{16})\\
&\times B(-1,\sm)^{-1} (\sl-1)(\sm^r\sl+\sm^{-r})\\
&\times (\sl - 2 \sl \sm^4 - 3 \sl \sm^8 + 2 \sl \sm^{12} - \sm^{16} + 6 \sl \sm^{16} - \sl^2 \sm^{16} + 2 \sl \sm^{20} - 3 \sl \sm^{24}\\
& \qquad - 2 \sl \sm^{28} + \sl \sm^{32})\end{array}
\end{equation}
which, up to a factor of an element in $\complex(\sm)$, is equal to the $A$-polynomial of $C$.
\subsection{Minimal degree of the recurrence relation}
In this section, we prove that the operator $R(t,\sm,\sl)$ given in the previous section has minimal $\sl$-degree by showing that no annihilator of $J_{\sc,n}$ has $\sl$-degree less than 4. 

By Lemma \ref{lemma: L>2}, a recurrence relation for $J_{\sc,n}$ has $\sl$-degree at least 2.

Suppose the recurrence polynomial of $C$ has $\sl$-degree 2 or 3. Then there are relatively prime Laurent polynomials $D_0, \ldots, D_3$ in $t$ and $\sm$ ($D_3$ possibly 0) such that
$$ D_3 J_{\sc,n+3} + D_2 J_{\sc,n+2} + D_1 J_{\sc,n+1} + D_0 J_{\sc,n} = 0.$$
Using equation (\ref{eq: s=2 J_C(n+1)}) to reduce $J_{\sc,n+2}$ and $J_{\sc,n+3}$,
\begin{align*}
0 &= D_3 J_{\sc,n+3} + D_2 J_{\sc,n+2} + D_1 J_{\sc,n+1} + D_0 J_{\sc,n}\\
&= D_3 (-t^{-4r(n+2)-2r}J_{\sc,n+2}+ t^{-2r(n+2)}J_{E,2n+5}) + D_2 J_{\sc,n+2} + D_1 J_{\sc,n+1} + D_0 J_{\sc,n}\\
&=  D_3t^{-2rn-4r}J_{\se,2n+5} 
+ (D_2-D_3 t^{-4rn-10r}) (-t^{-4r(n+1)-2r}J_{\sc,n+1}+ t^{-2r(n+1)}J_{\se,2n+3})
\\&\quad + D_1 J_{\sc,n+1} + D_0 J_{\sc,n}\\
&=  D_3t^{-2rn-4r}J_{\se,2n+5} + (D_2t^{-2rn-2r}-D_3 t^{-6rn-12r})J_{\se,2n+3}\\
&\quad + (D_1-D_2t^{-4rn-6r}+D_3 t^{-8rn-16r}) (-t^{-4rn-2r}J_{\sc,n}+ t^{-2rn}J_{\se,2n+1}) + D_0 J_{\sc,n}\\
&=  D_3t^{-2rn-4r}J_{\se,2n+5} + (D_2t^{-2rn-2r}-D_3 t^{-6rn-12r})J_{\se,2n+3}\\
&\quad + (D_1t^{-2rn}-D_2t^{-6rn-6r}+D_3 t^{-10rn-16r})J_{\se,2n+1}\\
&\quad+ (D_0-D_1t^{-4rn-2r}+D_2t^{-8rn-8r}-D_3 t^{-12rn-18r})J_{\sc,n},
\end{align*}
and since $Q_2 J_{\se,2n+5} + Q_1 J_{\se,2n+3} + Q_0 J_{\se,2n+1} = B$, we have $J_{\se,2n+5} = \frac{B}{Q_2} - \frac{Q_1}{Q_2} J_{\se,2n+3} - \frac{Q_0}{Q_2} J_{\se,2n+1}$, and so
\begin{align*}
0 &=  D_3t^{-2rn-4r}(\frac{B}{Q_2} - \frac{Q_1}{Q_2} J_{\se,2n+3} -\frac{Q_0}{Q_2} J_{\se,2n+1}) + (D_2t^{-2rn-2r}-D_3 t^{-6rn-12r})J_{\se,2n+3}\\
&\quad + (D_1t^{-2rn}-D_2t^{-6rn-6r}+D_3 t^{-10rn-16r})J_{\se,2n+1}\\
&\quad+(D_0-D_1t^{-4rn-2r}+D_2t^{-8rn-8r}-D_3 t^{-12rn-18r})J_{\sc,n}\\
&=  D_3\frac{B}{Q_2}t^{-2rn-4r}  +
(D_2t^{-2rn-2r}-D_3 t^{-6rn-12r}- D_3 \frac{Q_1}{Q_2} t^{-2rn-4r})J_{\se,2n+3} \\
&\quad + (D_1t^{-2rn}-D_2t^{-6rn-6r}+D_3 t^{-10rn-16r}- D_3\frac{Q_0}{Q_2}t^{-2rn-4r})J_{\se,2n+1}\\
&\quad +(D_0-D_1t^{-4rn-2r}+D_2t^{-8rn-8r}-D_3 t^{-12rn-18r})J_{\sc,n},
\end{align*}
and multiplying everything by $Q_2$, we have
$$0 = D_0' + D_1' J_{\se,2n+3} + D_2' J_{\se,2n+1} + D_3' J_{\sc,n}$$
where the $D_i'$ are Laurent polynomials in $t$ and $\sm$ given by
\begin{align*}
D_0' &= B D_3 t^{-2rn-4r},\\
D_1' &= Q_2 D_2 t^{-2rn-2r}-Q_2 D_3 t^{-6rn-12r}- Q_1 D_3 t^{-2rn-4r},\\
D_2' &= Q_2 D_1 t^{-2rn}-Q_2 D_2 t^{-6rn-6r}+Q_2 D_3 t^{-10rn-16r}- Q_0 D_3 t^{-2rn-4r},\\
D_3' &= Q_2 D_0- Q_2 D_1 t^{-4rn-2r}+Q_2 D_2 t^{-8rn-8r}-Q_2 D_3 t^{-12rn-18r}.
\end{align*}
Notice that if $D_0' = D_1' = D_2' = D_3' = 0$, then it follows that $D_0 = D_1 = D_2 = D_3 = 0$ as well. If the recurrence polynomial has $\sl$-degree 2, then $D_3 = 0$ and thus $D_0' = 0$. The following lemma rules out this possibility.

\begin{lemma} \label{lemma: s=2 not degree 2}
When $r > 8$ or $r < -8$, if $D_0' = 0$, then $D_i' = 0$ for $i = 1,2,3$ as well.
\end{lemma}

\begin{proof}
Suppose $D_3' \neq 0$ and $r > 8$. Then the lowest degree in $t$ of $D_3' J_{\sc,n}$ is $\ell[D_3'] + \ell[J_{\sc,n}]$. This term must vanish in the sum, so it must be canceled by another nonzero term, hence there must be another $D_i' \neq 0$. Then we must have
$$ \ell[D_3']+\ell[J_{\sc,n}] = \ell[D_1' J_{\se,2n+3} + D_2' J_{\se,2n+1}] \geq \min( \ell[D_1']+\ell[J_{\se,2n+3}], \ell[D_2']+\ell[J_{\se,2n+1}] )$$
due to possible cancellation if $D_1'J_{\se,2n+3}$ and $D_2' J_{\se,2n+1}$ have the same lowest degree, and we consider $\ell[0] = \infty$.

For sufficiently large $n$, $\ell[D_i']$ is a linear function in $n$. By Lemma \ref{lemma: E Degree}, we have $\ell[J_{\se,2n+1}] = -16n^2-12n$ and $\ell[J_{\se,2n+3}] = -16n^2-48n-28$, and by Lemma \ref{lemma: C Degree}, we know $\ell[J_{\sc,n}] = -2rn^2+2r$. 

Suppose $\min( \ell[D_1']+\ell[J_{\se,2n+3}], \ell[D_2']+\ell[J_{\se,2n+1}] ) = \ell[D_2'] +\ell[J_{\se,2n+1}] $. Then for large enough $n$,
\begin{align*}
\ell[D_3']+\ell[J_{\sc,n}] &\geq \ell[D_2'] + \ell[J_{\se,2n+1}], \text{ so}\\
\ell[D_3'] - \ell[D_2'] &\geq \ell[J_{\se,2n+1}]-\ell[J_{\sc,n}] \\
&= -16n^2-12n-(-2rn^2+2r)\\
&= (-16+2r)n^2-12n-2r,
\end{align*}
and since $r > 8$, $-16+2r>0$, so since the right hand side is quadratic in $n$, it will eventually be larger than the left hand side, which is only linear in $n$. This is a contradiction.

Likewise, if $\min( \ell[D_1']+\ell[J_{\se,2n+3}], \ell[D_2']+\ell[J_{\se,2n+1}] ) = \ell[D_1'] +\ell[J_{\se,2n+3}]$, we reach the same contradiction and conclude that $D_3' = 0$. Similarly, if $r < -8$, we look at $\hbar[D_3' J_{\sc,n}]$ and also conclude $D_3' = 0$.

Now that $D_3' = 0$, we have $0 = D_1' J_{\se,2n+3} + D_2' J_{\se,2n+1}$. Suppose $D_1' \neq 0$. Then $D_2' \neq 0$ as well, since $J_{\se,2n+3}$ is not the zero function. Then we have a first order homogeneous recurrence relation of $J_{\se,2n+1}$. If we can show that $J_{\se,2n+1}$ satisfies the hypotheses of Lemma \ref{lemma: L>2}, we will arrive at a contradiction. We have some recurrence relation for $J_{\se,2n+1}$ and it is easy to see that $\hbar[J_{\se,2n+1}] - \ell[J_{\se,2n+1}]$ is quadratic in $n$. Recalling that $J_{\se,-n} = -J_{\se,n}$, we have
$$J_{\se,-2n+1} = -J_{\se,2n-1} = -J_{\se,2n+1-2},$$
so by Lemma \ref{lemma: L>2}, $J_{\se,2n+1}$ cannot have a first order homogeneous recurrence relation. Therefore, $D_1' = D_2' = D_3' = 0$, as needed.
\end{proof}

We now know that the recurrence polynomial of $C$ does not have degree 2.
\begin{lemma} \label{lemma: s=2 not degree 3}
When $r > 8$ or $r < -8$, we have $D_i' = 0$ for $i = 0,1,2,3$.
\end{lemma}
\begin{proof}
The proof that $D_3' = 0$ is the same as in the proof of Lemma \ref{lemma: s=2 not degree 2}, noting that the polynomial $D_0'$ has lowest degree in $t$ which is only linear in $n$, so $\ell[D_3' J_{\sc,n}] \neq \ell[D_0']$ since $\ell[J_{\sc,n}]$ is quadratic in $n$.

Since $D_3' = 0$, we have $0 = D_0' + D_1' J_{\se,2n+3} + D_2' J_{\se,2n+1}$. If $D_0' = 0$, then we are done by Lemma \ref{lemma: s=2 not degree 2}, so assume for the sake of contradiction that $D_0' \neq 0$. Then we have an inhomogeneous recurrence relation of $J_{\se,2n+1}$ of $\sl$-degree 1. This gives rise to a homogeneous recurrence of $\sl$-degree 2
\begin{align*}
0 &= (\sl - 1) D_0'^{-1}(D_1'\sl+D_2')J_{\se,2n+1}\\
&= (\sl D_0'^{-1} D_1' \sl + \sl D_0'^{-1} D_2' -D_0'^{-1}D_1'\sl - D_0'^{-1} D_2')J_{\se,2n+1}\\
&= ( D_0'^{-1}(t,t^2\sm) D_1'(t,t^2\sm) \sl^2 +  (D_0'^{-1}(t,t^2\sm) D_2'(t,t^2\sm) -D_0'^{-1}(t,\sm)D_1'(t,\sm))\sl \\
&\quad - D_0'^{-1}(t,\sm) D_2'(t,\sm))J_{\se,2n+1},
\end{align*}
and multiplying on the left by $D_0'(t,\sm) D_0'(t,t^2\sm)$ gives
$$0 = ( D_0'(t,\sm) D_1'(t,t^2\sm) \sl^2 +  (D_0'(t,\sm) D_2'(t,t^2\sm) -D_0'(t,t^2\sm)D_1'(t,\sm))\sl - D_0'(t,t^2\sm) D_2'(t,\sm))J_{\se,2n+1}.$$
Recalling that $J_{\se,2n+1} = (\sm^r\sl+t^{-2r}\sm^{-r})J_{\sc,n}$, we get a homogeneous recurrence of $J_{\sc,n}$ with the annihilator
\begin{align*}
S(t,\sm,\sl) &:= ( D_0'(t,\sm) D_1'(t,t^2\sm) \sl^2 +  (D_0'(t,\sm) D_2'(t,t^2\sm) -D_0'(t,t^2\sm)D_1'(t,\sm))\sl \\
&\qquad - D_0'(t,t^2\sm) D_2'(t,\sm)) (\sm^r\sl+t^{-2r}\sm^{-r}).
\end{align*}
We want to evaluate at $t=-1$, but it is possible that some $D_i'(-1,\sm) = 0$. Then for $i = 0,1,2$, we have $D_i'(t,\sm) = (1+t)^{k_i} D_i''(t,\sm)$ for some Laurent polynomial $D_i''(t,\sm)$ and minimal $k_i \geq 0$ such that $D_i''(-1,\sm) \neq 0$. Hence
\begin{align*}
0 &= ( (1+t)^{k_0+k_1}D_0''(t,\sm) D_1''(t,t^2\sm) \sl^2 +  (1+t)^{k_0+k_2}D_0''(t,\sm) D_2''(t,t^2\sm)\sl \\
  &\qquad -(1+t)^{k_0+k_1}D_0''(t,t^2\sm)D_1''(t,\sm)\sl - (1+t)^{k_0+k_2}D_0''(t,t^2\sm) D_2''(t,\sm))(\sm^r\sl+t^{-2r}\sm^{-r}) J_{\sc,n}\\
&=(1+t)^{k_0}( (1+t)^{k_1}D_0''(t,\sm) D_1''(t,t^2\sm) \sl^2 +  (1+t)^{k_2}D_0''(t,\sm) D_2''(t,t^2\sm)\sl \\
&\qquad -(1+t)^{k_1}D_0''(t,t^2\sm)D_1''(t,\sm)\sl - (1+t)^{k_2}D_0''(t,t^2\sm) D_2''(t,\sm))(\sm^r\sl+t^{-2r}\sm^{-r}) J_{\sc,n},
\end{align*}
and we can cancel $(1+t)^{k_0}$ and likewise any other common factors if $k_1 > 0$ and $k_2 > 0$. Therefore we can assume without loss of generality that $k_0 = 0$ and at least one of $k_1$ or $k_2 = 0$. We check the cases.

\begin{case}
$k_1 = 0$ and $k_2 = 0$: 
\end{case}
Evaluating the annihilator $S$ at $t = -1$, we have
\begin{align*}
S(-1,\sm,\sl) &= ( D_0''(-1,\sm) D_1''(-1,\sm) \sl^2 +  D_0''(-1,\sm) D_2''(-1,\sm)\sl \\
&\qquad -D_0''(-1,\sm)D_1''(-1,\sm)\sl - D_0''(-1,\sm) D_2''(-1,\sm))(\sm^r\sl+\sm^{-r})\\
&= D_0''(-1,\sm)(\sl-1)(\sm^r \sl + \sm^{-r}) (D_1''(-1,\sm) \sl + D_2''(-1,\sm)).
\end{align*}
Since the recurrence polynomial of $C$ has $\sl$-degree 3 by assumption, we know that $S(t,\sm,\sl)$ is the recurrence polynomial up to a factor of a Laurent polynomial in $t$ and $\sm$, so over the field $\complex(\sm)$, $S(-1,\sm,\sl)$ must divide our degree 4 homogenous annihilator $R(t,\sm,\sl)$ of $J_{\sc,n}$ valued at $t = -1$ found in equation \ref{eq: s=2 AJ check}. This means we must have $D_1''(-1,\sm) \sl + D_2''(-1,\sm)$ divides 
$$\sl - 2 \sl \sm^4 - 3 \sl \sm^8 + 2 \sl \sm^{12} - \sm^{16} + 6 \sl \sm^{16} -\sl^2 \sm^{16} + 2 \sl \sm^{20} - 3 \sl \sm^{24} - 2 \sl \sm^{28} + \sl \sm^{32}.$$
But this is irreducible over $\complex(\sm)$, so $D_1''(-1,\sm) \sl + D_2''(-1,\sm) = 0$, thus $D_1''(-1,\sm) = 0$, which is a contradiction.

\begin{case}
$k_1 = 0$ and $k_2 > 0$:
\end{case}
Here, we have
\begin{align*}
S(-1,\sm,\sl) &= ( D_0''(-1,\sm) D_1''(-1,\sm) \sl^2 -D_0''(-1,\sm)D_1''(-1,\sm)\sl)(\sm^r\sl+\sm^{-r})\\
&= D_0''(-1,\sm)D_1''(-1,\sm) (\sl) (\sl-1) (\sm^r\sl+\sm^{-r}),
\end{align*}
and so $\sl$ must divide the irreducible factor of our degree 4 annihilator of $J_{\sc,n}$, which is a contradiction.

\begin{case}
$k_1 > 0$ and $k_2 = 0$:
\end{case}
This time,
$$S(-1,\sm,\sl) = D_0''(-1,\sm) D_2''(-1,\sm) (\sl - 1) (\sm^r\sl+\sm^{-r})$$
which has $\sl$-degree 2. Since $S(t,\sm,\sl)$ is, up to a factor of a Laurent polynomial in $t$ and $\sm$, the recurrence polynomial of $C$,  there must be some $P(t,\sm,\sl)$ in $\wT$ such that $R(t,\sm,\sl) = P(t,\sm,\sl) S(t,\sm,\sl)$, so $P$ has $\sl$-degree 1. However, $R(-1,\sm,\sl) = P(-1,\sm,\sl) S(-1,\sm,\sl)$, where $R$ has $\sl$-degree 4 while the right-hand side of the equation has $\sl$-degree 3. This is a contradiction.

In each case, we arrive at a contradiction, and conclude that $D_0' = D_1' = D_2' = D_3' = 0$.
\end{proof}
Therefore, our annihilator $R(t,\sm,\sl)$ is of minimal degree and hence the recurrence polynomial of $J_{\sc,n}$ up to normalization when $r > 8$ or $r < -8$.

\section{Case: $s > 2$}
We prove the $s > 2$ case of Theorem \ref{main result} in the same three steps as before.
\subsection{The Annihilator} \label{subsec: s>2 annihilator}
Recall equation (\ref{eq: s>2 peel}):
$$(t^{2rs}\sm^{rs} \sl^2-t^{-2rs}\sm^{-rs})J_{\sc,n} = t^{2r}\sm^{r}J_{\se,s(n+1)+1} - t^{-2r}\sm^{-r}J_{\se,s(n+1)-1}.$$
Define the sequence $T_n$ to be the right hand side of this equation:
\begin{equation} \label{eq: T_n definition}
T_n = t^{2r}\sm^{r}J_{\se,s(n+1)+1} - t^{-2r}\sm^{-r}J_{\se,s(n+1)-1}.
\end{equation}
Then to find a recurrence relation of $J_{\sc,n}$, it is enough to find an inhomogenous recurrence relation for $T_n$.
\begin{prop}\label{prop: s>2 annihilator}
There exists a polynomial $Q(t,\sm,\sl)$ of $\sl$-degree 2 which satisfies $Q(t,\sm,\sl) T_n = B(t,\sm)$ for some $B(t,\sm) \in \complex(t, \sm)$.
\end{prop}
\begin{proof}
Any second order inhomogeneous recurrence relation for $T_n$ looks like
\begin{align*}
B(t,\sm) &= \sum_{i=0}^2 Q_i(t,\sm) T_{n+i} \\
&= Q_0(t,\sm)(t^{2r}\sm^{r}J_{\se,s(n+1)+1} - t^{-2r}\sm^{-r}J_{\se,s(n+1)-1})\\
&\quad + Q_1(t,\sm)(t^{4r}\sm^{r}J_{\se,s(n+2)+1} - t^{-4r}\sm^{-r}J_{\se,s(n+2)-1})\\
&\quad + Q_2(t,\sm)(t^{6r}\sm^{r}J_{\se,s(n+3)+1} - t^{-6r}\sm^{-r}J_{\se,s(n+3)-1}).
\end{align*}
Recall that we have an inhomogeneous recurrence relation for $J_{\se,n+i}$ given in equation (\ref{eq: E Inhomogeneous}):
$$
\sum_{i=0}^2 P_i(t,t^{2n}) J_{\se,n+i} = b(t,t^{2n}).
$$
By substituting $s(n+1)-1+k$, with $0 \leq k \leq 2s$, for $n$ in equation (\ref{eq: E Inhomogeneous}), we shift the relation to
$$\sum_{i=0}^2 P_i(t,t^{2(s(n+1)-1+k)}) J_{\se,s(n+1)-1+k+i} = b(t,t^{2(s(n+1)-1+k)}),$$
or changing $t^{2n}$ to $\sm$,
$$\sum_{i=0}^2 P_i(t,t^{2(s-1+k)} \sm^s) J_{\se,s(n+1)-1+k+i} = b(t,t^{2(s-1+k)} \sm^s).$$
As in the $s=2$ case, we take a linear combination over $\complex(t,\sm)$ of these relations and aim to solve
\begin{equation}\label{eq: s>2 want to solve}
 \sum_{k=0}^{2s}c_k \sum_{i=0}^2 P_i(t,t^{2(s-1+k)}\sm^s) J_{\se,s(n+1)-1+k+i} = \sum_{i = 0}^2 Q_i(t,\sm) T_{n+i}
\end{equation}
for the unknown coefficients $c_0,\ldots, c_{2s}$ and $Q_0,Q_1,Q_2$. With these in hand, we will have
\begin{align*}
\sum_{i = 0}^2 Q_i(t,\sm) T_{n+i} &= \sum_{k=0}^{2s}c_k \sum_{i=0}^2 P_i(t,t^{2(s-1+k)}\sm^s) J_{\se,s(n+1)-1+k+i}\\
&= \sum_{k=0}^{2s}c_k b(t,t^{2(s-1+k)} \sm^s) \in \complex(t, \sm),
\end{align*}
so $\sum_{i = 0}^2 Q_i(t,\sm) \sl^i$ will be a polynomial giving a recurrence relation for $T_n$.

To solve equation (\ref{eq: s>2 want to solve}), we replace $T_{n+i}$ with its definition and see that this equation is equivalent to
\begin{align*} 
0 &= \sum_{k=0}^{2s}c_k \left(\sum_{i=0}^2 P_i(t,t^{2(s-1+k)}\sm^s) J_{\se,s(n+1)-1+k+i}\right) - \sum_{i = 0}^2 Q_i(t,\sm) T_{n+i} \\
&=  \sum_{k=0}^{2s}c_k \left(\sum_{i=0}^2 P_i(t,t^{2(s-1+k)}\sm^s) J_{\se,s(n+1)-1+k+i}\right) \\
&\qquad - \sum_{i = 0}^2 Q_i(t,\sm) (t^{(1+i)2r}\sm^{r}J_{\se,s(n+1+i)+1} - t^{-(1+i)2r}\sm^{-r}J_{\se,s(n+1+i)-1}).
\end{align*}
By setting the coefficients of each $J_{\se,s(n+1)-1+k}$ equal to zero, we obtain a linear system of equations over the field $\complex(t, \sm)$, with $2s+3$ equations (the indices range from $s(n+1)-1$ to $s(n+3)+1$) and $2s+4$ unknowns. We obtain a $(2s+3) \times (2s+4)$ matrix of coefficients, where we shorten the notation $P_i(t,\sm)$ to $P_i(\sm)$:
\begin{equation} \label{eq:The Matrix}
A =\resizebox{.9\textwidth}{!}{$\left[\begin{array}{ccccccc|ccc}
P_0(t^{2s-2}\sm^s) & 0 & 0 & \cdots & 0 &0 &0 & 0 & 0 & t^{-2r}\sm^{-r}\\
P_1(t^{2s-2}\sm^s) & P_0(t^{2s}\sm^s) & 0 & \cdots & 0 & 0 & 0& 0 & 0 & 0\\
P_2(t^{2s-2}\sm^s) & P_1(t^{2s}\sm^s) & P_0(t^{2s+2} \sm^s) & \cdots & 0 & 0 & 0& 0 & 0 & -t^{2r}\sm^{r}\\
\vdots &\vdots &\vdots &\ddots &\vdots & \vdots & \vdots & \vdots & \vdots & \vdots\\
0 & 0 & 0 & \cdots & P_2(t^{6s-6}\sm^s) & P_1(t^{6s-4} \sm^s) & P_0(t^{6s-2} \sm^s)& t^{-6r}\sm^{-r} & 0 & 0\\
0 & 0 & 0 & \cdots & 0 & P_2(t^{6s-4} \sm^s) & P_1(t^{6s-2} \sm^s)& 0 & 0 & 0\\
0 & 0 & 0 & \cdots & 0 & 0 & P_2(t^{6s-2} \sm^s) & -t^{6r}\sm^{r} & 0 & 0\\
\end{array}\right]$}\end{equation}
where the columns are arranged corresponding to the order $(c_0,c_1,\ldots,c_{2s},Q_2,Q_1,Q_0)$, the rows are arranged in order of increasing index of $J_{\se,s(n+1)-1+k}$, and the second-to-last column, corresponding to $Q_1$, contains $t^{-4r}\sm^{-r}$ and $-t^{4r}\sm^r$ in the $s+1$ and $s+3$ positions respectively and is zero everywhere else. We claim that this matrix has rank $2s+3$, which will guarantee that our polynomial has degree 2 rather than degree 1.

We wish to row-reduce $A$ enough to get a nonzero entry in the $(2s+3,2s+3)$ position. It is enough to do this setting $t = -1$. We can accomplish this in $2s+1$ steps, using the row operations
$$\mathcal{O}_i = \left\{ R_{i+1} \mapsto R_{i+1} - \frac{P_1(\sm^s)}{P_0(\sm^s)} R_i, R_{i+2} \mapsto R_{i+2} - \frac{P_2(\sm^s)}{P_0(\sm^s)}R_i \right\}$$
for $1 \leq i \leq 2s+1$. Performing these in succession leaves us with the following matrix (where we abbreviate $P_i(-1,\sm^s)$ to $P_i(\sm^s)$):
$$B = \resizebox{.9\textwidth}{!}{$\left[\begin{array}{ccccccc|ccc}
P_0(\sm^s) & 0 & 0 & \cdots & 0 &0 &0           & 0 & 0 & \sm^{-r}\\
0 & P_0(\sm^s) & 0 & \cdots & 0 & 0 & 0         & 0 & 0 & \frac{-P_1(\sm^s)}{P_0(\sm^s)} \sm^{r}\\
0 & 0 & P_0(\sm^s) & \cdots & 0 & 0 & 0         & 0 & 0 & -\sm^{r}-\frac{P_2(\sm^s)}{P_0(\sm^s)}\sm^{-r}\\
\vdots &\vdots &\vdots &\ddots &\vdots & \vdots & \vdots & \vdots & \vdots & \vdots\\
0 & 0 & 0 & \cdots & 0 & 0 & P_0(\sm^s)         & \sm^{-r} & x_{2s+1} & y_{2s+1}\\
0 & 0 & 0 & \cdots & 0 & 0 & 0                  & \frac{-P_1(\sm^s)}{P_0(\sm^s)} \sm^{-r} & x_{2s+2} & y_{2s+2}\\
0 & 0 & 0 & \cdots & 0 & 0 & 0                  & -\sm^{r}-\frac{P_2(\sm^s)}{P_0(\sm^s)}\sm^{-r} & x_{2s+3} & y_{2s+3}\\
\end{array}\right]$}$$
For some $x_i$'s and $y_i$'s, which we will discuss in a moment. It is easy to verify that $P_1(-1,\sm^s) \neq 0$, so we can further reduce this to
$$\resizebox{.95\textwidth}{!}{$\left[\begin{array}{ccccccc|ccc}
P_0(\sm^s) & 0 & 0 & \cdots & 0 &0 &0           & 0 & 0 & \sm^{-r}\\
0 & P_0(\sm^s) & 0 & \cdots & 0 & 0 & 0         & 0 & 0 & \frac{-P_1(\sm^s)}{P_0(\sm^s)} \sm^{-r}\\
0 & 0 & P_0(\sm^s) & \cdots & 0 & 0 & 0         & 0 & 0 & -\sm^{r}-\frac{P_2(\sm^s)}{P_0(\sm^s)}\sm^{-r}\\
\vdots &\vdots &\vdots &\ddots &\vdots & \vdots & \vdots & \vdots & \vdots & \vdots\\
0 & 0 & 0 & \cdots & 0 & 0 & P_0(\sm^s)         & 0 & x_{2s+1} + \frac{P_0(\sm^s)}{P_1(\sm^s)} x_{2s+2} & y_{2s+1}'\\
0 & 0 & 0 & \cdots & 0 & 0 & 0                  & \frac{-P_1(\sm^s)}{P_0(\sm^s)} \sm^{-r} & x_{2s+2} & y_{2s+2}\\
0 & 0 & 0 & \cdots & 0 & 0 & 0                  & 0 & x_{2s+3}' & y_{2s+3}'\\
\end{array}\right]$}$$
where $x_{2s+3}' = x_{2s+3} -\frac{P_0(\sm^s)}{P_1(\sm^s)}(\sm^{2r}+\frac{P_2(\sm^s)}{P_0(\sm^s)})x_{2s+2}$ (and similarly define $y_{2s+3}'$). We now claim $x_{2s+3}' \neq 0$, which implies that this matrix has rank $2s+3$.

In the matrix $B$ above, we denoted the entries of the $2s+2$ column as $x_1, \ldots, x_{2s+3}$. We have $x_i = 0$ for $1 \leq i \leq s$ and $x_{s+1} = \sm^{-r}$. Notice that due to the row operations, we have 
\begin{align*}
x_{s+2} &= \frac{-P_1(\sm^s)}{P_0(\sm^s)}\sm^{-r},\\
x_{s+3} &= -\sm^{r} + \frac{P_1(\sm^s)^2}{P_0(\sm^s)^2}\sm^{-r} - \frac{P_2(\sm^s)}{P_0(\sm^s)} \sm^{-r},\\
x_i &= \frac{P_2(\sm^s)}{P_0(\sm^s)}x_{i-2} - \frac{P_1(\sm^s)}{P_0(\sm^s)}x_{i-1}, \quad s+4 \leq i \leq 2s+3 
\end{align*}
For a nonzero Laurent polynomial $f \in \complex[t,\sm]^*$, define $\mu(f)$ to be the maximum degree in $\sm$ of $f$, and extend to rational functions $f/g \in \complex(t,\sm)^*$ by defining $\mu(f/g) = \mu(f) - \mu(g)$. This is well-defined and satisfies:
\begin{itemize}
\item $\mu(f \cdot g) = \mu(f) + \mu(g)$. 
\item $\mu(f+g) = \max(\mu(f),\mu(g))$ if $\mu(f) \neq \mu(g)$.
\end{itemize}
In particular, we have $\mu(P_0(1,\sm^s)) = \mu(P_2(1,\sm^s)) = 8s$ and $\mu(P_1(1,\sm^s)) = 12s$. We will show that $x_{2s+3}' \neq 0$ by showing $\mu(x_{2s+3}')$ is defined.

We have $\mu(x_{s+1}) = -r$, $\mu(x_{s+2}) = \mu(P_1(\sm^s)) - \mu(P_2(\sm^s)) + \mu(\sm^{-r}) = 4s-r$, $\mu(x_{s+3}) = \max(r,8s-r)$, and $\mu(x_i) = \max(\mu(x_{i-2}),4s+\mu(x_{i-1}))$. Let us examine the cases.

\begin{case} $r < 4s$: \end{case}
Since $r < 4s$, we have $\mu(x_{s+3}) = 8s-r$. Then $\mu(x_{s+4}) = \max(\mu(x_{s+2}), 4s+\mu(x_{s+3})) = 12s-r$, and by induction, $\mu(x_i) = (i-s-1)4s-r$ for $i \geq s+3$. Therefore
\begin{align*}
\mu(x_{2s+3}') &= \mu(x_{2s+3} -\frac{P_0(\sm^s)}{P_1(\sm^s)}(\sm^{2r}+\frac{P_2(\sm^s)}{P_0(\sm^s)})x_{2s+2})\\
&= \max(\mu(x_{2s+3}),\max(2r-4s+\mu(x_{2s+2}),-4s+\mu(x_{2s+2})))\\
&= \max((s+2)4s-r, 2r-4s+(s+1)4s-r, -4s+ (s+1)4s-r)\\
&= \max(4s^2+4s-r, 4s^2+r, 4s^2-r),
\end{align*}
which is defined because the entries cannot be equal. Therefore, $x_{2s+3}' \neq 0$.

\begin{case} $r > 4s$: \end{case}
Here, we have $\mu(x_{s+3}) = r$. Then $\mu(x_{s+4}) = \max(4s-r,4s+r) = 4s+r$ since $s > 0$. Then by induction, $\mu(x_{i}) = (i-s-3)4s+r$ for $i \geq s+3$. Therefore
\begin{align*}
\mu(x_{2s+3}') &= \mu(x_{2s+3} -\frac{P_0(\sm^s)}{P_1(\sm^s)}(\sm^{2r}+\frac{P_2(\sm^s)}{P_0(\sm^s)})x_{2s+2})\\
&= \max(\mu(x_{2s+3}),\max(2r-4s+\mu(x_{2s+2}),-4s+\mu(x_{2s+2})))\\
&= \max(4s^2+r, 2r-4s+(s-1)4s+r, -4s+ (s-1)4s+r)\\
&= \max(4s^2+r, 4s^2-8s-r, 4s^2-8s+r),
\end{align*}
which is defined because the entries cannot be equal. Therefore, $x_{2s+3}' \neq 0$.

We conclude that the matrix $A$ has rank $2s+3$ for all $s > 2$ and $r$ relatively prime to $s$. Therefore, we have a single free variable, say $Q_0$, which determines an annihilator for $T_n$. Let $Q(t,\sm,\sl)$ be the solution obtained this way such that its coefficients are relatively prime and in $\z[t^{\pm 1},\sm]$ and let $B(t,\sm) := \sum_{k=0}^{2s}c_k b(t,t^{2(s-1+k)} \sm^s)$.
\begin{exa}
We illustrate explicitly for the case $s = 3$. Setting $Q_0 = 1$, we have the matrix equation
$$\resizebox{\textwidth}{!}{$\begin{bmatrix}
P_0(t^4 \sm^3)  & 0 & 0 & 0 & 0&0 & 0&0 & 0 \\
 P_1(t^4 \sm^3) & P_0(t^6 \sm^3) & 0 & 0&0 & 0 & 0&0 & 0  \\
P_2(t^4 \sm^3)  & P_1(t^6 \sm^3) & P_0(t^8 \sm^3) &0& 0 & 0 & 0 & 0& 0  \\
 0 & P_2(t^6 \sm^3) & P_1(t^8 \sm^3) & P_0(t^{10} \sm^3) &0& 0 & 0&0  & t^{-4r} \sm^{-r} \\
 0 & 0 & P_2(t^8 \sm^3) & P_1(t^{10} \sm^3) & P_0(t^{12}\sm^3) &0& 0&0  & 0  \\
 0 & 0 & 0 & P_2(t^{10} \sm^3) & P_1(t^{12}\sm^3) & P_0(t^{14}\sm^3)&0&0  & -t^{4r} \sm^r  \\
 0 & 0 & 0 & 0 & P_2(t^{12}\sm^3) & P_1(t^{14}\sm^3)&P_0(t^{16}\sm^3)& t^{-6r} \sm^{-r}&0 \\
 0 & 0 & 0 & 0 & 0 & P_2(t^{14}\sm^3)&P_1(t^{16}\sm^3)&0  & 0  \\
 0 & 0 & 0 & 0 & 0 & 0&P_2(t^{16}\sm^3)& -t^{6r} \sm^{r}&0
\end{bmatrix} \begin{bmatrix} c_0\\c_1\\c_2\\c_3\\c_4\\c_5\\c_6\\Q_2\\Q_1\end{bmatrix} = \begin{bmatrix} -t^{-2r}\sm^{-r}\\0\\t^{2r}\sm^r\\0\\0\\0\\0\\0\\0\end{bmatrix}$}$$
Applying Cramer's rule gives us an annihilator of $T_n$:
$$Q(t,\sm,\sl) = 1+\frac{\det A_{9}}{\det A} \sl + \frac{\det A_{8}}{\det A} \sl^2,$$
where $A_i$ is the matrix $A$ with the $i$-th column replaced by the vector on the right side of the equality. This solution exists since $\det A \neq 0$. We can get an annihilator with coefficients in $\z[t^{\pm 1},\sm]$ by multiplying by a suitable element to clear the denominators. 
\end{exa}
It remains to check that $Q(t,\sm,\sl)$ is an inhomogeneous recurrence relation for $T_n$ rather than homogeneous. Suppose $Q(t,\sm,\sl)T_n = 0$. Then we have a homogeneous annihilator for $J_{\sc,n}$ since
$$Q(t,\sm,\sl)(t^{2rs}\sm^{rs} \sl^2-t^{-2rs}\sm^{-rs})J_{\sc,n} = 0.$$
It is proved in \cite{Le} that any recurrence polynomial of the colored Jones polynomial of a knot, when evaluated at $t = -1$, must contain the factor $(\sl-1)$. However, here we have
$$Q(-1,\sm,\sl)(\sm^{rs}\sl^2 - \sm^{-rs}) = Q(-1,\sm,\sl)\sm^{rs}(\sl-\sm^{-rs})(\sl+\sm^{rs}),$$
and $rs \neq 0$, so we must have $\sl-1$ dividing $Q(-1,\sm,\sl)$. We shall see shortly, however, that $Q(-1,\sm,\sl)$ is an irreducible polynomial of $\sl$-degree 2 over $\complex(\sm,\sl)$, so we cannot have $\sl-1$ dividing $Q(-1,\sm,\sl)$. We conclude that $Q(t,\sm,\sl) T_n = B(t,\sm)$ is an inhomogeneous recurrence relation.
\end{proof}

We would now like to check the $AJ$-conjecture. Recall that equation (\ref{A-poly of C}) gives
$$A_\sc(\sm,\sl)= (\sl-1)(\sm^{2rs}\sl^2-1) Red(Res_\fl(\frac{A_{\se}(\sm^{s},\fl)}{\fl-1},\fl^s-\sl))$$
up to a factor of a power of $\sm$. With $Q(t,\sm,\sl)$ given in Proposition \ref{prop: s>2 annihilator}, we have an annihilator
$$(\sl-1)B(t,\sm)^{-1}Q(t,\sm,\sl)(t^{2rs}\sm^{rs} \sl^2-t^{-2rs}\sm^{-rs})J_{\sc,n} = 0,$$
so to verify the $AJ$-conjecture, it is enough to show that $Q(-1,\sm,\sl)$ is equal to the remaining factor $Red(Res_\fl(\frac{A_\se(\sm^s,\fl)}{\fl-1},\fl^s-\sl))$ up to a factor of an element in $\complex(\sm)$. In fact, we shall see that $Res_\fl(\frac{A_\se(\sm^s,\fl)}{\fl-1},\fl^s-\sl)$ is irreducible for all positive values of $s$, so there are no repeated factors, so we can ignore the function $Red$. Connecting these two polynomials is the focus of the next section.

\subsection{The Resultant} \label{subsec: Resultant} The goal of this section is to prove the following proposition.
\begin{prop}\label{prop: Q=Res}
Let $r,s$ be relatively prime integers, $s > 1$, and let $\tilde{\alpha}_\se(t,\sm,\sl)$ be the inhomogeneous annihilator of $J_{\se,n}$ given in equation (\ref{eq: E Inhomogeneous}). Let $Q(t,\sm,\sl)$ be the polynomial given by Proposition \ref{prop: s>2 annihilator}. Then $$Q(-1,\sm,\sl) = C(\sm) Res_\fl(\tilde{\alpha}_\se(-1,\sm^s,\fl), \fl^s - \sl)$$ for some $C(\sm) \in \complex(\sm)$.
\end{prop}
Our method of attack is to show that in a commutative setting, the analogous problem to what we solved in the previous section has a straightforward solution via the resultant. Our brute force approach then reduces to more or less computing the resultant when we evaluate at $t = -1$.

Let us recall the definition of the resultant.
\begin{defin}
Let $\mathbb{K}$ be a field and let $f(x) = f_n x^n + \cdots + f_0$ and $g(x) = g_m x^m + \cdots + g_0$ be polynomials in $\mathbb{K}[x]$ of degree $n$ and $m$ respectively. Then the \textbf{resultant of $f$ and $g$} is the determinant of the $(m+n) \times (m+n)$ \textbf{Sylvester matrix of $f$ and $g$},
$$ Res(f,g) = \begin{array}{|cccccccc|}
f_0 & & & & g_0 & & &\\
f_1 & f_0& & & g_1 & g_0 & &\\
f_2 & f_1& \ddots & & g_2 & g_1& \ddots &\\
\vdots & f_2 & \ddots & f_0& \vdots & g_2 & \ddots &g_0\\
f_n & \vdots & \ddots& f_1 & g_m &\vdots &\ddots &g_1\\
 & f_n & \ddots & f_2 &  & g_m & \ddots & g_2\\
  & & \ddots & \vdots & & &\ddots & \vdots \\ 
  & & & f_n & & & & g_m\\[-4pt]
 \multicolumn{4}{c}{$\upbracefill$} & \multicolumn{4}{c}{$\upbracefill$}\\[-4pt]
 \multicolumn{4}{c}{\text{\small $m$ columns}} & \multicolumn{4}{c}{\text{ \small $n$ columns}}\\
\end{array}.$$
Moreover, given two polynomials over two variables, $f(x,y) = f_n(x)y^n + \ldots + f_0(x)$ and $g(x,y) = g_m y^m + \ldots + g_0(x)$, we define the \textbf{resultant of $f$ and $g$ with respect to $y$} or \textbf{eliminating $y$}, denoted $Res_y(f,g)$, to be $Res(f,g)$ over the field $\mathbb{K}(x)$.
\end{defin}
A key property of the resultant $Res_y(f,g)(x)$, as stated  in \cite[Ch.\ 12, p.\ 398]{GKZ}, is that $\alpha$ is a root of $Res_y(f,g)$ if and only if either $f(\alpha,y)$ and $g(\alpha,y)$ have a common root or $f_n(\alpha) = g_m(\alpha) = 0$. Thus, given a system of two polynomial equations in two variables
$$ \begin{cases} f(x,y) &= 0\\g(x,y) &= 0, \end{cases}$$
the resultant can be used to eliminate one of the variables from the system.

We now consider the commutative analog to the problem we solved in Section \ref{subsec: s>2 annihilator}. That is, given a recurrence relation of a sequence, how can we find a recurrence relation of a certain related sequence?

Fix a field $F := \complex(\sm)$ and consider $\mathcal{S} := \{S: \naturals \to F\}$, the set of $F$-valued sequences. Then $\mathcal{S}$ is an $F[\sl]$-module, where $\sl \cdot S_n := S_{n+1}$ for all $S \in \mathcal{S}$ and an element $c \in F$ acts on a sequence by multiplication. Given a sequence $S \in \mathcal{S}$, we can define the set $\mathcal{A}_{S} := \{ P(\sl) \in F[\sl] \mid P(\sl) \cdot S_n = 0 \text{ for all $n$}\}$, called the \textit{annihilator ideal} of $S$. Moreover, $F[\sl]$ is a principal ideal domain. A generator of the principal ideal $\mathcal{A}_{S}$ is thus an element of $\mathcal{A}_{S}$ of minimal degree and corresponds to a minimal order homogeneous recurrence relation of $S$.

Similarly, the set $\widehat{\mathcal{A}}_{S} := \{P(\sl) \in F[\sl] \mid \text{there exists $b \in F$ such that for every $n \in \naturals$, } P(\sl) \cdot S_n = b\}$ is a principal ideal, which consists of polynomials giving rise to inhomogeneous recurrence relations.

Suppose $S \in \mathcal{S}$ satisfies a minimal recurrence relation $\sum_{i=0}^d P_i S_{n+i} = b$. Then $S$ has an (inhomogeneous) annihilator $P(\sl) = \sum_{i=0}^d P_i \sl^i$. Consider the new sequence $T_n := S_{kn}$ for some $k > 1$. How do we find a recurrence relation for $T_n$? 

Let $\mathcal{S}_k := \{T \in \mathcal{S} \mid  T_n = S_{kn+i} \text{ for some $S_n  \in \mathcal{S}$, $i \in \zed$ }\}$. Now $\mathcal{S}_k$ is a $F[\sl,\fl]$-module, where $\sl \cdot T_{n} = S_{k(n+1)+i} = S_{kn+k+i} = T_{n+1}$ and $\fl \cdot T_{n} = S_{k(n+1/k)+i} = S_{kn+i+1} \in \mathcal{S}_k$. Then if $T_n = S_{kn}$ and $S_n$ is annihilated by $P(\sl)$, then $\sum_{i=0}^d P_i S_{kn+i}=b$ as well, so we have $T_n$ is annihilated by the polynomial $P(\fl)$. We would like to obtain an annihilator in the variable $\sl$ only, so we want to eliminate $\fl$. We can do this using the fact that $\fl^k$ acts as $\sl$. Then we effectively want to solve the system of polynomial equations
$$\begin{cases} P(\fl) &= 0\\ \fl^k-\sl &= 0. \end{cases}$$
Consider the polynomial $R(\sl) = Res_\fl(P(\fl),\fl^k-\sl)$, the resultant of the two polynomials with respect to the variable $\fl$. We check that $R(\sl)$ is actually what we want, an annihilator of $S_{kn}$.
\begin{lemma}\label{lemma: Res Annihilator}
Let $P(\sl)$ be a (possibly inhomogeneous) annihilator of the sequence $S_n$ and suppose $P(\sl)$ has no repeated roots in the algebraic closure $\overline{F}$. Then $R(\sl) = Res_\fl(P(\fl),\fl^k-\sl)$ is an annihilator of $S_{kn}$.
\end{lemma}
\begin{proof}
Let $\beta \in \overline{F}$ be a root of $P(\fl)$ and let $\alpha := \beta^k$. Then $\beta$ is a common root of $P(\fl)$ and $\fl^k-\alpha$. Then by the previously mentioned key property of resultants, $\alpha$ is a root of $R(\sl)$, and thus $\beta$ is a root of $R(\fl^k)$. Therefore every root of $P(\fl)$ is also a root of $R(\fl^k)$, so $P(\fl)$ divides $R(\fl^k)$ since $P(\fl)$ has no repeated roots. Since $P(\fl)$ annihilates $S_{kn}$, it follows that $R(\fl^k)$ is also an annihilator of $S_{kn}$. Since $\fl^k$ has the same action as $\sl$, we have $P(\sl)$ is an annihilator of $S_{kn}$, as needed.
\end{proof}
We will also require the following lemma.
\begin{lemma} \label{lemma: Res Irreducible}
Let $k > 1$ and let $P(\fl) \in F[\sl,\fl]$ be an irreducible polynomial of $\fl$-degree $d \leq 3$. If every root $\beta \in \overline{F}$ of $P$ satisfies $\beta^{k} \notin F$, then $R(\sl) = Res_\fl(P(\fl),\fl^k-\sl)$ is an irreducible polynomial of $\sl$-degree $d$.
\end{lemma}
\begin{proof}
Suppose $\alpha \in \overline{F}$ is a root of $R(\sl)$. Then there is $\beta \in \overline{F}$ such that $\beta^k = \alpha$ and $P(\beta) = 0$, and so $\alpha \notin F$ by assumption. Therefore, $R(\sl)$ has no roots in $F$, so it has no linear factors, and since $R(\sl)$ is also of degree $d \leq 3$ by the definition of the resultant, we conclude that $R(\sl)$ is irreducible.
\end{proof}
We are now ready to connect the polynomial $Q(t,\sm,\sl)$ of Proposition \ref{prop: s>2 annihilator} to the $A$-polynomial of $E$.
\begin{proof}[Proof of Proposition \ref{prop: Q=Res}.]
We know that the polynomial $R(\sl) := Res_\fl(\tilde{\alpha}_\se(-1,\sm^s,\fl),\fl^s-\sl)$ solves the problem of finding an annihilator for the sequence $S_{sn}$ if $S_n$ has the annihilator $\tilde{\alpha}_\se(-1,\sm^s,\sl)$. We will construct a sequence $S_n$ for which both $Q(-1,\sm,\sl)$ and this resultant are annihilators, and since $\widehat{\mathcal{A}}_{S}$ is a principal ideal, the resultant is irreducible, and the polynomials have the same degree, they will be the same up to a unit.

Fix the field $\complex(\sm)$. Since $\tilde{\alpha}_\se(-1,\sm^s,\sl)$ has $\sl$-degree 2, it gives rise to an inhomogeneous recurrence relation
$$P_2(-1,\sm^s) S_{n+2} + P_1(-1,\sm^s) S_{n+1} + P_0(-1,\sm^s) S_n = b(-1,\sm^s)$$
for some sequence $S_n$, and fixing two initial conditions defines $S_n$, so let $S_0 = 1$ and $S_1 = 1$.

Let $T_n = \sm^r S_{s(n+1)+1} - \sm^{-r} S_{s(n+1)-1}$ for some $r \in \zed$. By Lemma \ref{lemma: Res Annihilator}, we know that $R(\sl) S_{sn} = B(\sm)$ for some $B(\sm) \in \complex(\sm)$. Moreover, we have
\begin{align*}
R(\sl) \cdot T_n &= \sm^r R(\sl) \cdot S_{s(n+1)+1} - \sm^{-r} R(\sl) \cdot S_{s(n+1)-1} \\
&= \sm^r R(\sl) \fl^{s+1} \cdot S_{sn} - \sm^{-r} R(\sl) \fl^{s-1} \cdot S_{sn}\\ 
&= (\sm^r - \sm^{-r})B(\sm),
\end{align*}
so $R(\sl)$ is an annihilator of $T_n$.

We claim $T_n$ is not a constant sequence. We show this by computing $\mu(T_n)$, the $\sm$-degree of $T_n$, and showing it is not zero. First, we have
$$ S_{n+2} = - \frac{P_1(-1,\sm^s)}{P_2(-1,\sm^s)} S_{n+1} - \frac{P_0(-1,\sm^s)}{P_2(-1,\sm^s)} S_n + \frac{b(-1,\sm^s)}{P_2(-1,\sm^s)},$$
and $\mu(P_0(-1,\sm^s)) = \mu(P_2(-1,\sm^s)) = 8s$, $\mu(P_1(-1,\sm^s)) = 12s$, and $\mu(b(-1,\sm^s)) = 11s$, so
$$\mu(S_{n+2}) = \max(4s + \mu(S_{n+1}), \mu(S_n), 11s), \quad n \geq 0$$
while $\mu(S_0) = \mu(S_1) = 0$, and so $\mu(S_2) = 11s$ and $\mu(S_n) = 11s+(n-2)4s = 4sn +3s$ for $n \geq 3$. Thus
\begin{align*}
\mu(T_n) & = \max(r + 4s(s(n+1)+1)+3s, -r + 4s(s(n+1)-1)+3s) \\
&= \max(r+4s^2n+4s^2+7s, -r + 4s^2n+4s^2-s)
\end{align*}
since $r+4s^2n+4s^2+7s = -r + 4s^2n+4s^2-s$ if and only if $r = -4s$, but $r$ and $s$ are relatively prime. Since $\mu(T_n) \neq 0$ and is finite, $T_n$ is not a constant sequence. In particular, $\widehat{\mathcal{A}}_T \neq \complex(\sm)$.

Next, we claim that $R(\sl)$ is irreducible over $\complex(\sm)$ and consequently a generator of $\widehat{\mathcal{A}}_T$. By Lemma \ref{lemma: Res Irreducible}, it is enough to show that for every root $\beta \in \overline{\complex(\sm)}$ of $\tilde{\alpha}_\se(-1,\sm^s,\fl)$, $\beta^k \notin \complex(\sm)$.

Suppose $P_2(-1,\sm^s) \beta^2 +P_1(-1,\sm^s) \beta+P_0(-1,\sm^s) = 0$. Since $\tilde{\alpha}_\se(-1,\sm^s,\fl)$ is irreducible over $\complex(\sm)$, we know $\beta \notin \complex(\sm)$. Then 
$$\beta^2 = -\frac{P_1(-1,\sm^s)}{P_2(-1,\sm^s)} \beta - \frac{P_0(-1,\sm^s)}{P_2(-1,\sm^s)},$$
which is not in $\complex(\sm)$ since $P_1(-1,\sm^s) \neq 0$. We claim that for all $k \in \naturals$, $\beta^k = a_k \beta + b_k$ for some $a_k, b_k \in \complex(\sm)$, $a_k \neq 0$. We show this by computing $\mu(a_k)$. We have $\mu(a_1) = \mu(1) = 0$, $\mu(a_2) = \mu(P_1(-1,\sm^s))-\mu(P_2(-1,\sm^s)) = 4s$, $\mu(b_1) = \mu(0) = -\infty$, $\mu(b_2) = 0$. Let $k > 2$, and assume that for all $i < k$, $\mu(a_i) = 4s(i-1)$. Then we have
\begin{align*}
\beta^k &= (a_{k-1} \beta + b_{k-1})\beta\\
&= a_{k-1} \beta^2 + b_{k-1} \beta\\
&= a_{k-1} \left(-\frac{P_1(-1,\sm^s)}{P_2(-1,\sm^s)} \beta - \frac{P_0(-1,\sm^s)}{P_2(-1,\sm^s)}\right) + b_{k-1} \beta\\
&= \left( -\frac{P_1(-1,\sm^s)}{P_2(-1,\sm^s)}a_{k-1} + b_{k-1}\right)\beta -\frac{ P_0(-1,\sm^s)}{P_2(-1,\sm^s)} a_{k-1},
\end{align*}
which means $b_k = - \frac{P_0(-1,\sm^s)}{P_2(-1,\sm^s)} a_{k-1}$, so $\mu(b_k) = \mu(a_{k-1})$, hence $\mu(b_{k-1}) = \mu(a_{k-2})$. Thus
\begin{align*}
\mu(a_k) &= \mu\left(-\frac{P_1(-1,\sm^s)}{P_2(-1,\sm^s)}a_{k-1} + b_{k-1}\right) \\
&= \max(4s + \mu(a_{k-1}), a_{k-2})\\
&= \max(4s+4s(k-2),4s(k-3))\\
&= 4sk-4s,
\end{align*}
so by induction, $\mu(a_k) > 0$ for all $k$ and hence $\beta^k$ is never in $\complex(\sm)$. We conclude that $R(\sl)$ is irreducible.

Now, consider this alternate solution to the same problem, in the manner of the proof of Proposition \ref{prop: s>2 annihilator}. We know $\tilde{\alpha}_\se(-1,\sm^s,\fl) \cdot S_{sn}=b(-1,\sm^s)$, so we have the equations
\begin{align*}
\tilde{\alpha}_\se(-1,\sm^s,\fl) \cdot S_{s(n+1)-1+j} &= \sum_{i=0}^2 P_i(-1,\sm^s) S_{s(n+1)-1+i+j}\\
 &= b(-1,\sm^s)
\end{align*}
for any $j$. We suspect we can find an inhomogeneous annihilator $\hat{Q}_2 \sl^2 + \hat{Q}_1 \sl + \hat{Q}_0$ of $T_n$ by finding a suitable linear combination of these equations:
$$ \sum_{j = 0}^{2s} c_j \sum_{i=0}^2 P_i(-1,\sm^s) S_{s(n+1)-1+i+j} = \sum_{i=0}^2 \hat{Q}_i T_{n+i}.$$
Setting the coefficients of each $S_{sn+k}$ equal to zero, we get a linear system in $2s+3$ equations and $2s+4$ unknowns. The resulting matrix of coefficients is exactly the matrix $A$ in equation \ref{eq:The Matrix} after setting $t = -1$. By the same argument as before, we see that we can find nonzero $\hat{Q}_0$, $\hat{Q}_1$, and $\hat{Q}_2$, so $\hat{Q}(\sl) = \hat{Q}_2 \sl^2 + \hat{Q}_1 \sl + \hat{Q}_0$ is an inhomogeneous annihilator of $T_n$ of degree 2, and this $\hat{Q}(\sl)$ is exactly $Q(-1,\sm,\sl)$ up to a factor of a rational function in $\complex(\sm)$.

Then since $R(\sl)$ is a generator of $\widehat{\mathcal{A}}_T$, $R(\sl)$ divides $Q(-1,\sm,\sl)$, and since they are both of degree 2, we have $C(\sm) R(\sl) = Q(-1,\sm,\sl)$ for some $C(\sm) \in \complex(\sm)$, which completes the proof.
\end{proof}

Finally, since $E$ satisfies the $AJ$-conjecture, we know $\tilde{\alpha}_\se(-1,\sm,\sl) = f(\sm) \frac{A_\se(\sm,\sl)}{\sl-1}$ for some polynomial $f(\sm)$. The resultant being the determinant of the Sylvester matrix gives us 
$$Res_\fl(\tilde{\alpha}_\se(-1,\sm^s,\fl),\fl^s-\sl) = Res_\fl(f(\sm^s) \frac{A_\se(\sm^s,\fl)}{\fl-1},\fl^s-\sl) = f(\sm^s)^s Res_\fl( \frac{A_\se(\sm^s,\fl)}{\fl-1},\fl^s-\sl).$$ 
It follows that $Q(-1,\sm,\sl)$ is equal to $Res_\fl( \frac{A_\se(\sm^s,\fl)}{\fl-1},\fl^s-\sl)$ up to a factor of an element in $\complex(\sm)$, and consequently the $AJ$-conjecture is satisfied for the knot $C$.

\subsection{Minimal degree of the recurrence relation}

Suppose there is a recurrence relation of $\sl$-degree at most 4 of $J_{\sc,n}$
$$D_4 J_{\sc,n+4} + D_3 J_{\sc,n+3} + D_2 J_{\sc,n+2} + D_1 J_{\sc,n+1} + D_0 J_{\sc,n} = 0.$$
Combining equation (\ref{eq: s=2 J_C(n+1)}) with the definition of $T_n$ in equation (\ref{eq: T_n definition}), we have
$$J_{\sc,n+2} = t^{-4rs} \sm^{-2rs} J_{\sc,n} + t^{-2rs} \sm^{-rs} T_n,$$
so we can use this to simplify our recurrence:
\begin{align*}
0 &= D_4 J_{\sc,n+4} + D_3 J_{\sc,n+3} + D_2 J_{\sc,n+2} + D_1 J_{\sc,n+1} + D_0 J_{\sc,n}\\
&= D_4(t^{-12rs} \sm^{-2rs} J_{\sc,n+2} + t^{-6rs} \sm^{-rs} T_{n+2}) + D_3(t^{-8rs} \sm^{-2rs} J_{\sc,n+1} + t^{-4rs} \sm^{-rs} T_{n+1})\\
&\quad + D_2(t^{-4rs} \sm^{-2rs} J_{\sc,n} + t^{-2rs} \sm^{-rs} T_n) + D_1J_{\sc,n+1} + D_0 J_{\sc,n}\\
&= (D_4 t^{-16rs}\sm^{-4rs} + D_2 t^{-4rs}\sm^{-2rs} + D_0) J_{\sc,n} + (D_3 t^{-8rs}\sm^{-2rs} + D_1) J_{\sc,n+1}\\
&\quad +(D_4 t^{-6rs}\sm^{-rs})T_{n+2} + (D_3 t^{-4rs} \sm^{-rs})T_{n+1} + (D_4t^{-14rs}\sm^{-3rs}+D_2t^{-2rs}\sm^{-rs})T_n,
\end{align*}
and by Proposition \ref{prop: s>2 annihilator}, we have
$$T_{n+2} = \frac{B}{Q_2} - \frac{Q_1}{Q_2} T_{n+1} - \frac{Q_0}{Q_2} T_n,$$
so making this substitution yields
\begin{align*}
0 &= (D_4 t^{-16rs}\sm^{-4rs}+D_2 t^{-4rs}\sm^{-2rs} + D_0) J_{\sc,n}+ (D_3 t^{-8rs}\sm^{-2rs} + D_1) J_{\sc,n+1} \\
&\quad + (D_4(t^{-14rs}\sm^{-3rs} - \frac{Q_0}{Q_2}t^{-6rs}\sm^{-rs}) + D_2 t^{-2rs}\sm^{-rs})T_n \\
&\quad  + (D_3 t^{-4rs} \sm^{-rs} - D_4 \frac{Q_1}{Q_2} t^{-6rs}\sm^{-rs})T_{n+1} + D_4 \frac{B}{Q_2} t^{-6rs}\sm^{rs},
\end{align*}
and multiplying both sides of the equation by $Q_2$ gives us something of the form
$$0 =  D_4' J_{\sc,n} + D_3' J_{\sc,n+1} + D_2' T_n + D_1' T_{n+1} + D_0'$$
where the $D_i'$ are Laurent polynomials in $t$ and $\sm$. It is easy to see that $D_4' = \ldots = D_0' = 0$ implies $D_4 = \ldots = D_0 = 0$. Notice that if the degree of the recurrence polynomial is less than 4, then $D_4 = 0$ and so $D_0' = 0$.
\begin{lemma}
When $r > 4s$ or $r < -4s$, if $D_0' = 0$, then $D_i' = 0$ for $i = 1,\ldots, 4$ as well.
\end{lemma}
\begin{proof}
Suppose $r > 4s$ and $D_4' \neq 0$. As in the $s=2$ case, we compare the lowest degrees in $t$ of the summands. We need another $D_i'$ to be nonzero in order to cancel $D_4' J_{\sc,n}$, so we examine the cases.

Notice that $\ell[T_n] = \min(2r(n+1)+\ell[J_{\se,s(n+1)+1}], -2r(n+1)+\ell[J_{\se,s(n+1)-1}])$, so by Lemma \ref{lemma: E Degree},
$$\ell[T_n] =  -2r(n+1)-4s^2n^2+(10s-8s^2)n-4s^2+10s-4.$$
We know $\ell[J_{\sc,n}]$ has a coefficient of $-rs$ on $n^2$ by Lemma \ref{lemma: C Degree}, and $-rs < -4s^2$, so the lowest degree $D_4' J_{\sc,n}$ cannot be canceled by the lowest degree in $D_2' T_n$ or $D_1' T_{n+1}$, since $\ell[D_i']$ is only linear in $n$ for sufficiently large $n$. Then we must have $D_3' \neq 0$. But this gives us
\begin{align*}
\ell[D_4'] - \ell[D_3'] &= \ell[J_{\sc,n+1}] - \ell[J_{\sc,n}]\\
&= -2rs n -rs +(-1)^n(s-2)(4s-r-2),
\end{align*}
which is eventually linear on the left but alternating on the right, which is impossible. Therefore, $D_4' = 0$. Similarly, $D_3' = 0$, and when $r < -4s$, using the highest degree in $t$ gives $D_4'= D_3' = 0$ as well.

Next, given that $D_4' = D_3' = 0$, we have
$$0 = D_2' T_n + D_1' T_{n+1}.$$
Suppose $D_2' \neq 0$ and $D_1' \neq 0$. Then we have a first order homogeneous recurrence relation for $T_n$. We find a contradiction using Lemma \ref{lemma: L>2}. It is not hard to see that the breadth $\hbar[T_n] - \ell[T_n]$ is quadratic in $n$. Also, notice that
\begin{align*}
T_{-n} &= t^{2r(-n+1)} J_{\se,s(-n+1)+1} - t^{-2r(-n+1)} J_{\se,s(-n+1)-1}\\
&=  - t^{-2r(n-1)} J_{\se,s(n-1)-1} + t^{2r(n-1)} J_{\se,s(n-1)+1}\\
&= T_{n-2},
\end{align*} 
so by Lemma \ref{lemma: L>2}, any homogeneous recurrence relation of $T_n$ has order at least 2, which is a contradiction. Therefore one of $D_2'$ or $D_1'$ is zero, but since $T_{n}$ is not zero, we must have $D_1' = D_2' = 0$, which completes the proof.
\end{proof}
This implies that the recurrence polynomial of $C$ has $\sl$-degree 4.
\begin{lemma}
When $r > 4s$ or $r < -4s$, we have $D_i' = 0$ for $i = 0,\ldots, 4$.
\end{lemma}

\begin{proof}
Noting that $\ell[D_0']$ is linear in $n$, the proof that $ D_4' = D_3' = 0$ is the same.

Since $D_4' = D_3' = 0$, we have
$$0 = D_2' T_n + D_1' T_{n+1} + D_0'.$$
This is a degree 1 inhomogeneous recurrence relation for $T_n$. The rest of the proof is analagous to the proof of Lemma \ref{lemma: s=2 not degree 3} in the $s=2$ case.
\end{proof}
We conclude that
$$(\sl-1)B(t,\sm)^{-1} Q(t,\sm,\sl)(t^{2rs}\sm^{rs} \sl^2-t^{-2rs}\sm^{-rs})$$
is the recurrence polynomial of the $(r,s)$-cabled knot $C$ over the figure eight knot if $s > 2$ and $r > 4s$ or $r < -4s$ up to a factor of an element in $\complex(t,\sm)$.

\end{document}